\documentclass[12pt]{article}
\usepackage{amsthm}
\usepackage{varioref}
\usepackage{amsmath,amstext,amsthm,a4,amssymb,amscd}   
\usepackage[T1]{fontenc}
\usepackage[utf8]{inputenc}
\usepackage{lmodern}
\usepackage{array}
\usepackage{latexsym}
\usepackage{fancyhdr}
\usepackage{mathrsfs}\let\cal\mathscr
\usepackage[all]{xy}
\usepackage{makeidx}   
\usepackage{color}
\usepackage{pifont}
\usepackage{amsfonts,amssymb}
\usepackage{hyperref}
\usepackage{cleveref}
\usepackage[numbers,square]{natbib}
\usepackage[disable]{todonotes}
\usepackage{verbatim}
\usepackage{relsize}
\usepackage{dsfont}

\usepackage{avant}
\usepackage[T1]{fontenc}      

\usepackage{amsfonts,amssymb}  
 
\usepackage{graphicx}

\usepackage{natbib}
\newcommand \Om {\Omega}
\newcommand \om {\omega}
\newcommand \0 {\emptyset}

\renewcommand \i {\sqrt{-1}}
\renewcommand \leq {\leqslant}
\renewcommand \geq {\geqslant}

\newcommand{\pttheta}{p^{\frac{\epsilon}{2}}}
\newcommand{\eptt}{\epsilon p^{\frac{\epsilon}{2}}}
\newcommand{\ptheta}{p^{\frac{-1+\epsilon}{2}}}
\newcommand{\ept}{\epsilon p^{\frac{-1+\epsilon}{2}}}

\DeclareMathOperator{\Vol}{Vol}
\DeclareMathOperator{\End}{End}

\DeclareMathOperator{\Supp}{Supp}

\newcommand \< {\mathcal{h}}
\renewcommand \> {\mathcal{i}}

\newcommand \cinf {\CC^\infty}

\newcommand \Id {{\rm Id}}

\renewcommand \epsilon {\varepsilon}

\newcommand \CC {{\cal C}}

\newcommand \HH {{\cal H}}

\newcommand \NN {{\mathcal N}}

\newcommand \PP {{\cal P}}

\newcommand \Dk[2] {\frac{\partial^{#1}}{\partial {{#2}^{#1}}}}
\newcommand \Dkk[3] {\frac{\partial^{#1} #2}{\partial {{#3}^{#1}}}}

\newcommand \R {\mathbb R}
\newcommand \IP {\mathbb P}
\newcommand \IE {\mathbb E}
\newcommand \IV {\mathbb V}

\newcommand \C {\mathbb C}
\newcommand \IH {\mathbb H}
\newcommand \ID {\mathbb D}
\newcommand \N {\mathbb N}

\newcommand \Z {\mathbb Z}

\newcommand \IT {\mathbb T}

\newcommand \fl {\rightarrow}

\newcommand \ignore[1] {}

\theoremstyle{plain}
\newtheorem{theorem}{Theorem}[section]
\newtheorem{lem}[theorem]{Lemma}
\newtheorem{cor}[theorem]{Corollary}
\newtheorem{prop}[theorem]{Proposition}
\theoremstyle{definition}
\newtheorem*{ackn*}{Acknowledgments}
\newtheorem{defi}[theorem]{Definition}
\newtheorem{ex}[theorem]{Example}

\newtheorem{rmk}[theorem]{Remark}

\numberwithin{equation}{section}

\crefname{equation}{}{}
\crefname{lem}{Lemma}{Lemmas}
\crefname{theorem}{Theorem}{Theorems}
\crefname{cor}{Corollary}{Corollaries}
\crefname{ex}{Example}{Examples}
\crefname{defi}{Definition}{Definitions}
\crefname{prop}{Proposition}{Propositions}
\crefname{section}{Section}{Sections}
\crefname{subsection}{Section}{Sections}
\crefname{rmk}{Remark}{Remarks}
\crefname{nota}{Notation}{Notations}

\hypersetup{
    colorlinks=true,
    linkcolor=red,
    filecolor=magenta,      
    urlcolor=cyan,
}

\begin{document}

\title{\bf{Partial Bergman kernels and determinantal point processes
on Kähler manifolds}}
\author{Louis IOOS}
\date{}
\maketitle
\newcommand{\Addresses}{{
  \bigskip
  \footnotesize

  \textsc{CY Cergy Paris Université, 95300 Pontoise,
France}\par\nopagebreak
  \textit{E-mail address}: \texttt{louis.ioos@cyu.fr}
  
}}

\begin{abstract}
We compute the full off-diagonal asymptotics of the equivariant and
partial Bergman kernels
associated with a circle action on a prequantized Kähler manifold with bounded geometry at infinity,
then use these results to compute the asymptotics of the
linear statistics of the associated determinantal point process as the number
of points grows to infinity,
showing that its distribution converges to a centered normal
variable with variance given by the sum of an $H^1$-norm squared in the bulk
and an $H^{\frac{1}{2}}$-norm squared on the boundary of the associated droplet.
\end{abstract}

\section{Introduction}

Given a measured space $(X,dv_X)$ and a finite orthonormal
family $\{\psi_j\in L^2(X,\C)\}_{j=1}^N$ with $N\in\N$, the
associated \emph{determinantal point process} is the measure
$d\nu_N$ on the $N$-fold product $X^N$ defined for all
$(x_1,\cdots,x_N)\in X^N$ by
\begin{equation}\label{DPPdef}
d\nu_N(x_1,\dots,x_N):=\frac{1}{N!}\,\left|\det(\psi_j(x_k))_{j,k=1}^N\right|^2\,dv_X(x_1)\cdots dv_X(x_N)\,.
\end{equation}
As explained for instance in \cite[Lem.\,4.5.1]{HKPV09},
this measure
defines in fact a probability measure over $X$.
Determinantal point processes were introduced by Macchi in \cite{Mac75} as
general models of a probability distribution on configurations
of $N$ points over $X$ exhibiting a repulsive behavior, since
the determinant in \cref{DPPdef}
vanishes as soon as any two points of the configuration coincide,
while also experiencing a confining potential,
since square-integrable functions
typically tend to $0$ at infinity when $\Vol(X,dv_X)=+\infty$.
In order to describe the distribution of typical configurations
with respect to $d\nu_N$,
it is natural to consider the associated
\emph{linear statistics} with respect to 
a test function $f\in L^\infty(X,\R)$, which are defined as the random variable
$\NN[f]:X^{N}\to\R$ given for any $(x_1,\dots,x_N)\in X^N$ by
\begin{equation}\label{linstatintro}
\NN[f](x_1,\dots,x_N):=\sum_{j=1}^N\,f(x_j)\,,
\end{equation}
and study its behavior in the \emph{thermodynamic limit} as $N\to+\infty$.
Specifically, a sequence of measures of the form \cref{DPPdef}
for each $N\in\N$
is said to
admit an \emph{equilibrium measure} $\nu$ over $X$ if
the linear statistics \cref{linstatintro} satisfy the following convergence in probability,
\begin{equation}\label{LLNintro}
\frac{1}{N}\NN[f]\xrightarrow{~N\to\infty~}~
\frac{1}{\Vol(X,d\nu)}\int_{X}\,f\,d\nu\,.
\end{equation}
This property is called a \emph{law of large numbers}.
The support $D:=\Supp\mu\subset X$ of the equilibrium measure is called
the \emph{droplet},
and the law of large numbers \cref{LLNintro}
shows in particular that typical configurations tend to
accumulate in the droplet as $N\to\infty$.
Their fluctuations, on the other hand, are usually described by a
\emph{central limit theorem} for the linear statistics \cref{linstatintro},
whose variance should then be described in terms of
the variations of $f$ over $D\subset X$.

In the context of this paper,
the fundamental example of such a determinantal point process
indexed by $N\in\N$
is the so-called \emph{Ginibre ensemble},
which corresponds to  $X=\C$ equipped with
the Lebesgue measure $dv_\C$, together with the orthonormal family
$\{\psi_j^{(N)}\in L^2(\C,\C)\}_{j=1}^N$ defined for
all $1\leq j\leq N$
and all $z\in\C$ by
\begin{equation}\label{Gindef}
\psi_j^{(N)}(z):=\sqrt{\frac{N^{j+1}}{\pi (j-1)!}}~z^{j-1}\,e^{-\frac{N}{2}|z|^2}\,.
\end{equation}
This determinantal point process has been introduced by Ginibre in
\cite[\S\,1]{Gin65}, who showed that it describes the distribution of eigenvalues of
a random matrix of size $N\times N$ with entries
following independent complex centered Gaussians with variance $\frac{1}{N}$.
As explained for instance in \cite[\S\,1.2.2]{Ser24},
it can also be interpreted as the Boltzman-Gibbs distribution
for a \emph{Coulomb gas} in the plane confined by a quadratic potential
at inverse temperature $\beta=2$.
The main result of Ginibre in \cite[\S\,1]{Gin65} then states that these
determinantal point processes admit
$d\nu:=\mathds{1}_{\mathbb{D}}\,dv_\C$
as an equilibrium measure as $N\to\infty$, where $\mathbb{D}\subset\C$ denotes the unit disk.
This is the celebrated \emph{circular law}
for the Ginibre ensemble.
Concerning the fluctuations, Rider and Virag established in \cite{RV07}
a central limit theorem for the Ginibre ensemble, showing that for
any compactly supported $f\in\CC^1_c(\C,\R)$, the centered
random variable $\NN[f]-\IE[\NN[f]]$ converges in distribution to a centered
normal random variable with variance
\begin{equation}\label{Varintro}
\lim\limits_{N\to\infty}\IV[\NN[f]]
=
\frac{1}{4\pi}\int_{\mathbb{D}}\,|df|^2\,dv_\C+
\sum_{k\in\Z}\,|k|\,
|\hat{f}_k|^2\,,
\end{equation}
where for any $k\in\Z$, we write
$\hat{f}_k\in\C$ for the $k$-th Fourier coefficient of $f\in\CC^1_c(\C,\R)$ restricted to the unit circle
$\partial\mathbb{D}\subset\C$. While the first term
is an homogeneous $H^1$-norm squared of $f$ restricted to
the droplet $\ID\subset\C$, the second term can be viewed as
an homogoneous $H^{\frac{1}{2}}$-norm squared
of $f$ restricted to its boundary $\partial\ID\subset\C$.
This result has been extended to more general potentials
over $\C$ with suitable growth at infinity
by Ameur, Hedenmalm and Makarov in \cite{AHM11} and Leblé and Serfaty
in \cite{LS18}.
As explained by Deleporte and Lambert in \cite[Th.\,1.2]{DL25}, the appearance
of an $H^{\frac{1}{2}}$-norm over the boundary of the droplet in \cref{Varintro}
can be understood as a manifestation of the universality
of the \emph{strong Szegö limit
theorem} established by Szegö in \cite{Sze52}.

In this paper, we extend these results to a much larger class of
measured spaces, where $X$ is a \emph{Kähler manifold} equipped with
its Riemannian volume form $dv_X$, in a general
set-up first introduced by Berman in \cite{Ber18}.
In this context, we consider a symplectic manifold $(X,\om)$
without boundary
together with a Hermitian line bundle $(L,h^L)$ over $X$ endowed with a
Hermitian
connection $\nabla^L$ satisfying the
following \emph{prequantization formula},
\begin{equation}\label{preq}
\om=\frac{\sqrt{-1}}{2\pi}R^L\,,
\end{equation}
where $R^L\in\Om^2(X,\C)$ is the curvature of $\nabla^L$.
We let also $X$ be equipped with an integrable
complex structure $J\in\cinf(X,\End(TX))$ compatible with $\om$,
making $(X,\om,J)$ into a \emph{Kähler manifold prequantized by}
$(L,h^L,\nabla^L)$.
We can then consider the associated \emph{Kähler metric} $g^{TX}$, which
is the Riemannian metric defined over $X$
by the formula
\begin{equation}\label{gTX}
g^{TX}:=\om(\cdot,J\cdot)\,,
\end{equation}
and let $dv_X$ be the associated Riemannian volume form.
These data naturally endow $(L,h^L)$ with a holomorphic structure for which
$\nabla^L$ is the associated \emph{Chern connection}.
For any $p\in\N$, write
$L^p:=L^{\otimes p}$ for the $p^{\text{th}}$-tensor power of $L$
equipped with the induced Hermitian metric $h^{L^p}$,
and consider the space $H^0_{(2)}(X,L^p)$
of square-integrable holomorphic sections of $L^p$
for the $L^2$-Hermitian product induced by $h^{L^p}$ and $dv_X$.
Given a finite orthonormal family $\{s_j\in H^0_{(2)}(X,L^p)\}_{j=1}^{N_p}$
with $N_p\in\N$,
we consider the measure
$d\nu_{N_p}$ over the $N_p$-fold product $X^{N_p}$ defined for all
$(x_1,\cdots,x_{N_p})\in X^{N_p}$ by
\begin{equation}\label{DPPholdef}
d\nu_{N_p}(x_1,\dots,x_{N_p}):=\frac{1}{N_p!}\,\left|\det(s_j(x_k))_{j,k=1}^{N_p}\right|^2_p\,dv_X(x_1)\cdots dv_X(x_{N_p})\,,
\end{equation}
where $|\cdot|_p$ denotes the Hermitian norm induced by $h^{L^p}$ on
$\bigotimes_{1\leq j\leq N_p}L^p_{x_j}$.

In case $X$ is compact,
the space of holomorphic sections
$H^0_{(2)}(X,L^p)=:H^0(X,L^p)$ is finite-dimensional for all $p\in\N$,
and the determinantal point processes \cref{DPPholdef} associated
with orthonormal bases of $H^0(X,L^p)$
have first been studied by Berman in \cite{Ber18}.
When $X$ is not necessarily compact,
a natural construction of finite orthonormal families in
$H^0_{(2)}(X,L^p)$ consists in considering
a compatible action of the circle $S^1$ on $(L,h^L,\nabla^L)$ over $(X,\om,J)$,
inducing a unitary representation of $S^1$
on $H^0_{(2)}(X,L^p)$ for each $p\in\N$. One can then consider the \emph{weight decomposition}
\begin{equation}\label{decwghtspace}
H^0_{(2)}(X,L^p)=\widehat{\bigoplus\limits_{m\in\Z}}\,H^0_{(2)}(X,L^p)_m\,,
\end{equation}
given by the $L^2$-orthogonal Hilbert direct sum
of the \emph{weight spaces}
$H^0_{(2)}(X,L^p)_m$ defined for all $m\in\Z$ by 
\begin{equation}\label{wghtspace}
H^0_{(2)}(X,L^p)_m:=\{s\in H^0_{(2)}(X,L^p)~|~
\varphi_{t}^*s=e^{2\pi\sqrt{-1} tm}s\,,\text{ for all }t\in\R\}\,,
\end{equation}
where $\varphi_{t}^*$ denotes the pullback
by the action of $S^1\simeq\R/\Z$ on $L$ over $X$
for all $t\in\R$.
As we explain in \cref{actionsec},
a compatible action of $S^1$ on $(L,h^L,\nabla^L)$ over $(X,\om,J)$
actually determines a \emph{moment map} $\mu\in\cinf(X,\R)$
for the Hamiltonian action of $S^1$ on the symplectic manifold $(X,\om)$,
called \emph{Kostant moment map}.
Under the assumption described in \cref{volgrowth}
that $\mu\in\cinf(X,\R)$ has \emph{polynomial
growth},
so that in particular it is proper and bounded from below,
we show in \cref{QR=0} that
for any $p\in\N$ large enough, the
subspace
\begin{equation}\label{Hpdef}
\HH_p:=\widehat{\bigoplus\limits_{m\leq 0}}\,H^0_{(2)}(X,L^p)_m
\subset H^0_{(2)}(X,L^p)\,,
\end{equation}
given by the Hilbert direct sum of weight spaces \cref{wghtspace}
with negative weights,
has finite dimension $N_p\in\N$, so that one can consider the
determinantal point process \cref{DPPholdef} associated with
an orthonormal basis $\{s_j\in \HH_p\}_{j=1}^{N_p}$.
As we explain in \cref{Ginibreex},
in the special case of $X=\C$ endowed with the action of $S^1\subset\C$
by multiplication, we recover in this way the Ginibre ensemble
\cref{Gindef}.

Assume also
that $S^1$ acts freely on the compact submanifold $\mu^{-1}(0)\subset X$,
and consider the associated \emph{symplectic reduction} $X_0:=\mu^{-1}(0)/S^1$,
with induced Riemannian metric $g^{TX_0}$ and Riemanian volume form $dv_{X_0}$.
For any $f\in\cinf(X,\C)$ and any $k\in\N$,
consider the function $|\hat{f}_k|^2:X_0\to\R$ induced by
its \emph{$k$-th Fourier coefficient}, defined for any $x\in\mu^{-1}(0)$ by
\begin{equation}\label{Fouriercoeff}
\hat{f}_k(x):=\int_0^1\,e^{2\pi \i t k}\,f(\varphi_t(x))\,dt\,.
\end{equation}
Under the additional assumption described in \cref{setting}
that the Kähler manifold $(X,\om,J)$ prequantized by $(L,h^l,\nabla^L)$
has \emph{bounded geometry at infinity},
the main result of this paper is the following.

\begin{theorem}\label{mainth}
Let $(X,\om,J)$ be a Kähler manifold prequantized by
$(L,h^L,\nabla^L)$
with bounded geometry at
infinity endowed with a compatible $S^1$-action
such that its Kostant moment
map $\mu\in\cinf(X,\R)$ has polynomial growth,
and assume that $S^1$ acts freely on $\mu^{-1}(0)\subset X$.
Then for any $f\in L^\infty(X,\R)$,
the linear statistics
$\NN_p[f]:X^{N_p}\to\R$ defined as in \cref{linstatintro}
satisfy the following convergence
in probability as $p\to\infty$,
\begin{equation}\label{LLNmainth}
\frac{1}{N_p}\NN_p[f]\xrightarrow{~p\to\infty~}~
\frac{1}{\Vol\left(\{\mu<0\}\right)}\int_{\{\mu<0\}}\,f\,dv_X\,.
\end{equation}
Furthermore, for any smooth $f\in\cinf_c(X,\R)$ with
compact support,
the variance of the associated
linear statistics satisfies
\begin{equation}\label{Varmainth}
\lim\limits_{p\to\infty}\,p^{-n+1}
\IV\left[\NN_p[f]\right]=
\frac{1}{4\pi}\int_{\{\mu<0\}}\,|df|^2\,dv_X(x)+
\frac{1}{2}\int_{X_0}\,\sum_{k\in\Z}\,|k|\,
|\hat{f}_k(x)|^2\,dv_{X_0}(x)\,,
\end{equation}
where $n:=\frac{\dim X}{2}$, and the random variable $N_p^{\alpha}(\NN_p[f]-\IE[\NN_p[f]])$
with $\alpha=\frac{1}{2n}-\frac{1}{2}$
converges in distribution to a centered
normal random variable with variance \cref{Varmainth} as $p\to\infty$.
\end{theorem}

\cref{mainth} thus states that  as $p\to\infty$, the determinantal point processes
\cref{DPPholdef} associated with the spaces
\cref{Hpdef} admit the equilibrium measure $d\nu:=\mathds{1}_{D}\,dv_X$
over $X$ with droplet $D:=\{\mu<0\}$,
and satisfy a central limit theorem with fluctuations given by the sum
of an homogeneous $H^1$-norm squared over
$D\subset X$ and
an homogeneous $H^{\frac{1}{2}}$-norm
squared over its boundary $\mu^{-1}(0)\subset X$
in the directions of the circle action, as in the
strong Szegö limit theorem \cite{Sze52}.
The proof of \cref{mainth} is described in \cref{DPPsec}.
In particular, the law of large numbers \cref{LLNmainth}
is a consequence of \cref{LLN} while the
asymptotics \cref{Varmainth} of the variance is established in \cref{Varth}, and those asymptotics imply
the central limit
theorem by a general argument from the theory of determinantal
point processes due to
Soshnikov in \cite[Th.\,1]{Sos02}.

In
the special case of $X=\C$ endowed with the action of $S^1\subset\C$
by multiplication,
the law of large numbers \cref{LLNmainth} in \cref{mainth} recovers
the classical circular law for the Ginibre ensemble
established by Ginibre in
\cite[\S\,1]{Gin65}, while the asymptotics
\cref{Varmainth} for the variance in \cref{mainth}
recover the asymptotics \cref{Varintro} established
by Rider and Virag in \cite{RV07}
together with the corresponding central limit theorem.
On the other hand, in the
case of the trivial $S^1$-action on a compact prequantized
Kähler manifold, 
one can arrange for the Kostant moment map
$\mu\in\cinf(X,\R)$ to be constant and strictly positive,
so that $\mu^{-1}(0)=\0$ and the hypotheses of \cref{mainth}
are trivially satisfied. Then
while the spaces $\HH_p=H^0(X,L^p)$ coincide
with the space of holomorphic sections for all $p\in\N$, and
\cref{mainth} recovers results of 
Berman in \cite[Th.\,1.4,\,1.5]{Ber18}.
Note that the $H^{\frac{1}{2}}$-term
in formula \cref{Varmainth} for the variance vanishes
in that case.

The hypotheses of
\cref{mainth} are satisfied
more generally in the case of a compact prequantized
with a compatible $S^1$-action
acting freely on $\mu^{-1}(0)\subset X$, extending the
results of Berman in \cite{Ber18} to include the case of
smooth functions whose support are not necessarily included
in the droplet $D\subset X$,
displaying the extra $H^{\frac{1}{2}}$-term in \cref{Varintro}
in the case $\mu^{-1}(0)\neq\0$.
Up to a standard shift of weights in \cref{Hpdef},
the assumptions of \cref{mainth} are also
satisfied in the important case when
$X$ is a \emph{toric manifold},
with circle action induced by the choice of $S^1\subset\IT^n$
inside the associated torus, which includes
in particular the case $X=\C^n$ for general $n\in\N$.
The law of large numbers \cref{LLNmainth} then
recovers a result of Berman in \cite[Th.\,3.4]{Ber09c}, while
\cref{mainth} establishes a central limit theorem extending
the
result of Rider and Virag in \cite{RV07} in all these cases.
This also includes the case of the real and complex hyperbolic spaces, the relevance of the associated
determinantal point processes being
studied for instance by Bufetov, Fan and Qiu in \cite{BFQ18} in case $X=\IH^n$ is the real hyperbolic space, and
by Bufetov and Qiu in \cite{BQ22} in case $X=\mathbb{B}^n$ is the complex
hyperbolic space.

The proof of \cref{mainth} is based
on the full asymptotic expansion of the \emph{partial Bergman kernel}
as $p\to\infty$ that we establish in \cref{partBergsec},
recovering in particular the results of Ross and Singer
in \cite[Th.\,1.2]{RS17}, Zelditch
and Zhou in \cite[(8),\,Th.\,4]{ZZ19}
and Shabtai in \cite[Th.\,1.7]{Sha25},
who establish the asymptotic expansion of the partial Bergman
kernel
in neighborhoods of size of order $\frac{1}{\sqrt{p}}$
around the boundary
$\mu^{-1}(0)\subset X$ and outside neighborhoods
of size of order $1$ as $p\to\infty$.
The full off-diagonal expansion over arbitrary neighborhoods of $\mu^{-1}(0)$ established in \cref{faroffdiagprop,partBergasyth} plays a crucial role in the proof of \cref{mainth}.
These results
are in turn based on the full off-diagonal
asymptotic expansion
of the \emph{equivariant Bergman kernel} associated with the weight
space \cref{wghtspace} for each $m\in\Z$, extending the analogous
results in the case $m=0$ of
Ma and Zhang in \cite[Th.\,0.2]{MZ08}.
The tools used in this paper are based on the full off-diagonal
asymptotic expansion of the Bergman kernel established
by Dai, Liu and Ma in \cite[Th.\,4.18']{DLM06} and its
extension to non-compact manifolds
established by
Ma and Marinescu in \cite[\S\,3.5]{MM08a}, which we recall
in \cref{asysec}.
A comprehensive introduction of this theory can be found in
their book \cite{MM07}.
Near-diagonal
asymptotic expansions 
for a larger class of
equivariant Bergman kernels in the case $m=0$
have also been
established by Paoletti in \cite[Th.\,1.2]{Pao12},
and of a larger class of
partial Bergman kernels by Coman and Marinescu in \cite{CM17}
and Zelditch and Zhou in \cite{ZZ19b}.

In the
setting of the trivial $S^1$-action on a compact prequantized Kähler manifold
originally considered by Berman in \cite{Ber18},
Charles and Estienne computed in
\cite[Cor.\,1.7]{CE20} the asymptotics of the linear statistics
with respect to the characteristic function $f=\mathds{1}_U$ of an open
subset $U\subset X$ with smooth boundary.
In particular, they establish a central limit
theorem in this case. On the other hand,
Berman also considers in \cite[Th.\,1.4,\,1.5]{Ber18}
the case of  a compact prequantized Kähler manifold
with singular Hermitian metric $h^L$,
in which case the droplet $D\subset X$
does not necessarily coincide with $X$,
and
computes the asymptotics of the linear statistics
with respect to $f\in L^\infty(X,\R)$
sufficiently regular and with support
strictly included in $D$, so that the second term
in formula \cref{Varmainth} still vanishes. 
We hope that the methods of this paper can be used to extend both of
these results to the case of a non-trivial $S^1$-action on a
not necessarily compact prequantized Kähler manifold.


\begin{ackn*}
The author learned the theory of determinantal point processes
through extended conversations with Pierre Lazag, and wishes to thank him first
and foremost. The author also wants to thank Razvan Apredoaei, Paul Dario,
Thibaut Lemoine and Xiaonan Ma for useful discussions.
This project was partially supported by the ANR-23-CE40-0021-01 JCJC
project QCM.
\end{ackn*}

\section{Bergman kernels and circle actions}

After describing the general setting of the
paper in \cref{setting}, we recall in \cref{asysec}
the results of Ma and Marinescu in \cite[Chap.\,6]{MM07}
on the full asymptotic expansion of the Bergman kernel
of prequantized Kähler manifolds with bounded geometry
which will constitute the fundamental tool of this paper,
then introduce in \cref{actionsec} the
Kostant moment map of a compatible $S^1$-action on a
prequantized Kähler manifold.

\subsection{Setting}
\label{setting}

Let $(X,\om)$ be a symplectic manifold of dimension $2n\in\N$
without boundary,
together with a Hermitian line bundle $(L,h^L)$ over $X$ endowed with a
Hermitian
connection $\nabla^L$ satisfying the
\emph{prequantization formula} \cref{preq}.
We also assume that $X$ is equipped with an integrable
complex structure $J\in\cinf(X,\End(TX))$ compatible with $\om$, making $(X,\om,J)$
into a \emph{Kähler manifold} and $(L,h^L)$ into a holomorphic
Hermitian line bundle of which $\nabla^L$ is the Chern connection.
We write $g^{TX}$ for the associated \emph{Kähler metric}
\cref{gTX} and $dv_X$ for the induced Riemannian volume form.
In this paper,
we will always make the assumption that these data have \emph{bounded geometry
at infinity},
meaning that $(X,g^{TX})$ is complete with positive injectivity radius
and that the derivatives at any order of
$R^L,\,J$ and $g^{TX}$ are uniformly bounded in the norms induced
by $h^L$ and $g^{TX}$. Note that this assumption is automatically
satisfied in the important case when $X$ is compact.

For any $p\in\N$, write $h^{L^p}$ and $\nabla^{L^p}$
for the Hermitian metric and connection on
the $p^{\text{th}}$-tensor power
$L^p:=L^{\otimes p}$ respectively induced 
by $h^L$ and $\nabla^L$ on $L$.
We denote by $\cinf_c(X,L^p)$ the space of
compactly supported smooth sections of $L^p$,
endowed with the \emph{$L^2$-Hermitian product}
$\<\cdot,\cdot\>_{p}$ given for any $s_1,s_2\in\cinf_c(X,L^p)$ by the formula
\begin{equation}\label{L2}
\<s_1,s_2\>_{p}:=\int_X h^{L^p}(s_1(x),s_2(x))\,dv_X(x)\,.
\end{equation}
Let
$L^2(X,L^p)$ denote the completion of $\cinf_c(X,L^p)$
with respect to the associated $L^2$-norm, and
write $H^0_{(2)}(X,L^p)\subset L^2(X,L^p)$
for the space
of $L^2$-holomorphic sections of $L^p$.
The following result is a consequence of standard elliptic theory,
and introduces the basic fundamental tool of this paper.

\begin{prop}\label{Bergprop}
{\cite[Rmk.\,1.4.3]{MM07}}
For any $p\in\N$, the orthogonal projection onto
$H^0_{(2)}(X,L^p)\subset L^2(X,L^p)$
with respect to the $L^2$-product \cref{L2}, denoted by
\begin{equation}\label{projdef}
P_p:L^2(X,L^p)\longrightarrow H^0_{(2)}(X,L^p)
\end{equation}
admits a smooth Schwartz kernel
$P_p(\cdot,\cdot)\in\cinf(X\times X,L^p \boxtimes (L^p)^*)$
with respect to $dv_X$, called the \emph{Bergman kernel},
characterized for any
$s\in\cinf(X,L^p)$ and $x\in X$ by
the formula
\begin{equation}\label{ker}
(P_p s)(x)=\int_X P_p(x,y).s(y)\,dv_X(y).
\end{equation}
\end{prop}
%

Finally, for any Riemannian manifold $(Y,g^{TY})$,
we will write $d^{Y}(\cdot,\cdot)$ for the associated distance over $Y$,
and for any $x\in Y$ and $\epsilon>0$, we will write
$B^Y(x,\epsilon)\subset Y$ for the geodesic ball of
center $x\in Y$ and radius $\epsilon>0$.

\subsection{Asymptotic expansion of the Bergman kernel}
\label{asysec}

Let us consider the setting described in \cref{setting},
and for any $r,\,p\in\N$, let $|\cdot|_{\CC^r}$ denote the local
$\CC^r$-norm on $L^p \boxtimes (L^p)^*$ induced by
$h^{L^p}$ and $\nabla^{L^p}$. 
The following result describes the off-diagonal decay of the Bergman kernel
introduced in \cref{Bergprop}.

\begin{theorem}{\cite[Th.\,1]{MM15}}
\label{theta}
There exists $c>0$ such that for any $r\in\N$, there is $C_r>0$
such that for all $p\in\N$ and $x,y\in X$, the following estimate holds,
\begin{equation}
|P_p(x,y)|_{\CC^r}\leq C_r p^{n+\frac{r}{2}}\,e^{-c\sqrt{p}\,d^X(x,y)}\,.
\end{equation}
\end{theorem}

For any $x\in X$, let $|\cdot|$ denote the Euclidean norm on
$T_xX$ and on $T_x^*X$ induced by $g^{T_xX}$, and for any subspace $E_x\subset T_xX$,
let $B^{E_x}(0,\epsilon_0)\subset E_x$ denote the open ball
in $E_x$ of radius $\epsilon_0>0$ with respect to the induced norm.
We write $B^{TX}(0,\epsilon_0)\subset TX$ for the ball bundle over $X$ whose
fibre over any $x\in X$ is given by $B^{T_xX}(0,\epsilon_0)\subset T_xX$.
Recall that since $(X,g^{TX})$
has bounded geometry
at infinity, its injectivity radius is bounded from below.

To describe asymptotic estimates for the Bergman kernel in a neighborhood
of the diagonal,
we will need the following definition.

\begin{defi}\label{chart}
Given $\epsilon_0>0$ smaller than the injectivity radius of $(X,g^{TX})$,
we say that a smooth map $\psi:B^{TX}(0,\epsilon_0)\to X$ is
a \emph{bounded family of charts} if for any $x\in X$,
its restriction $\psi_x:B^{T_xX}(0,\epsilon_0)\to X$
to $B^{T_xX}(0,\epsilon_0)\subset B^{TX}(0,\epsilon_0)$
is a diffeomorphism on its image $U_x\subset X$
satisfying $\psi_x(0)=x$ and $d\psi_{x,0}=\Id_{T_xX}$,
and if for any $r\in\N$,
there exists $C_r>0$ such that for all $x\in X$, we have
\begin{equation}
\left|\exp^{X}_x\circ\,\psi_x^{-1}\right|_{\CC^r}\leq C_r\,,
\end{equation}
where $\exp^{X}:B^{TX}(0,\epsilon_0)\fl X$ is the \emph{exponential map}
of $(X,g^{TX})$.
\end{defi}

Note that the the exponential map $\exp^{X}:B^{TX}(0,\epsilon_0)\fl X$
is itself a bounded family of charts by definition.
We fix one of them for the rest of the section. 

For any $x\in X$, identify $L$ over the image
$U_x\subset X$ of $\psi_x:B^{T_xX}(0,\epsilon_0)\to X$
with $L_x$ through parallel
transport with respect to $\nabla^{L}$ along radial lines
of $B^{T_xX}(0,\epsilon_0)$ and
pick a unit vector $e_x\in L_x$ to identify $L_x$ with $\C$.
Under the natural isomorphism $\End(L^p)\simeq\C$, the formulas below will not
depend on this identification.
For any $p\in\N$ and any kernel
$K_p(\cdot,\cdot)\in\cinf(X\times X,L^p\boxtimes (L^p)^*)$,
we write
\begin{equation}\label{Kxdef}
K_{p,x}(Z,Z')\in\C
\end{equation}
for its image in this trivialization evaluated at
$Z,Z'\in B^{T_xX}(0,\epsilon_0)$, and for any  smooth $f\in\cinf(X,\C)$,
we write $f_x:=\psi_x^*f$ for its pullback in these coordinates.
%
We then get the following immediate consequence of
\cref{theta}.
\begin{cor}\label{thetagal}
For any $r,\,k\in\N$ and $\epsilon>0$, there exists $C_k>0$ such that for all $p\in\N,\,x\in X,$ and $Z,Z'\in B^{T_xX}(0,\epsilon_0)$ satisfying $|Z-Z'|>\epsilon p^{\frac{-1+\epsilon}{2}}$, the following estimate holds,
\begin{equation}
|P_{p,x}(Z,Z')|_{\CC^r}\leq C_k\,p^{-k}.
\end{equation}
\end{cor}
We use the following explicit local model for the Bergman kernel
from \cite[(3.25)]{MM08b}, for any $x\in M$ and $Z,Z'\in T_xX$,
\begin{equation}\label{PPreal}
\PP_x(Z,Z'):=\exp\left(-\frac{\pi}{2}|Z-Z'|^2-\pi\sqrt{-1}\om_x(Z,Z')\right)\,.
\end{equation}
Let $|\cdot|_{\CC^r(X)}$ denote the $\CC^r$-norm
over the fibred product $B^{TX}(0,\epsilon_0)\times_X B^{TX}(0,\epsilon_0)$
in the direction of $X$ via the Levi-Civita
connection. We can now state the following fundamental result on the 
near diagonal expansion of the Bergman kernel,
which was first established by Ma and Marinescu in
\cite[\S\,3.5]{MM08a}
following Dai, Liu and Ma in \cite[Th.\,4.18']{DLM06}.
We present this
result in a simplified form
following the approach
of Lu, Ma and Marinescu in \cite[Th.\,2.1]{LMM16},
and which can be directly deduced from \cite[Problem\,6.1]{MM07}.

\begin{theorem}\label{asy}
There is a family $\{J_{r,x}(Z,Z')\in\C[Z,Z']\}_{r\in\N}$ of
polynomials in $Z,\,Z'\in T_xX$ depending smoothly on $x\in X$,
of the same parity as $r\in\N$ and of total degree less than $3r\in\N$,
such that for any $k,\,j,\,j'\in\N$ and
$\delta>0$, there is $\epsilon_0>0$ and $C>0$
such that for all $\epsilon\in\,]0,\epsilon_0[,\,x\in X$, $p\in\N$ and all $Z,Z'\in T_xX$ satisfying
$|Z|,|Z'|<\epsilon p^{\frac{-1+\epsilon}{2}}$, we have
\begin{multline}\label{asyfla}
\sup_{|\alpha|+|\alpha'|\leq j}\Big|\Dk{\alpha}{Z}\Dk{\alpha'}{Z'}\big(p^{-n}P_{p,x}(Z,Z')\\
-\sum_{r=0}^{k-1} J_{r,x}(\sqrt{p}Z,\sqrt{p}Z')\PP_x(\sqrt{p}Z,\sqrt{p}Z')\left.p^{-\frac{r}{2}}\big)\right|_{\CC^{j'}(X)}\leq Cp^{-\frac{k-j}{2}+\delta}.
\end{multline}
Furthermore, for all $x\in X$, we have $J_{0,x}\equiv 1$.
\end{theorem}

The following result is a direct consequence of \cref{theta,asy}.

\begin{cor}\label{asydiag}
For any $r\in\N$,
there exists $C_r>0$
such that for any $p\in\N$ and any $x,\,y\in X$, we have
\begin{equation}
\left|\,p^{-n}\,P_p(x,y)\,\right|_{\CC^r}\leq C_r
\quad\text{ and }\quad
\left|\,p^{-n}\,P_p(x,x)-1\,\right|_{\CC^r}\leq C_r\,p^{-1}\,.
\end{equation}
\end{cor}
%
%

\subsection{Circle actions on prequantized Kähler manifolds}
\label{actionsec}

Let $(X,\om)$ be a symplectic manifold
together with a Hermitian line bundle $(L,h^L)$ over $X$
with Hermitian connection $\nabla^L$ satisfying the
prequantization formula \cref{preq},
and assume that $L$ is endowed with
a compatible action of the circle $S^1$ lifting an action on $X$,
so that it preserves $h^L$ and $\nabla^L$.
Write $\varphi_t:X\to X$ for the
flow of diffeomorphisms of $X$
induced by the action of $S^1\simeq\R/\Z$ for all $t\in\R$,
and for any $p\in\N$, write
$\varphi_{t,p}:L^p\to L^p$ for its lift to $L^p$.
We consider the induced action by pullback on a smooth sections
$s\in\cinf(X,L^p)$ defined for all $x\in X$ by
\begin{equation}\label{action}
(\varphi_t^* s)(x):=\varphi_{t,p}^{-1}\,s(\varphi_t(x))\,.
\end{equation}
Let $\xi\in\cinf(X,TX)$ denote
the associated fundamental vector field over $X$, defined for all $x\in X$ by
\begin{equation}\label{xidef}
\xi_x:=\frac{d}{dt}\Big|_{t=0}\varphi_t(x)\,.
\end{equation}
The following definition is due to Kostant
in \cite[Th.\,4.5.1]{Kos70}.

\begin{defi}\label{Kostantmudef}
The \emph{moment map} $\mu\in\cinf(X,\R)$ of the action of
$S^1$ on $(L,h^L,\nabla^L)$ over $X$ is defined by the following formula,
for all $s\in\cinf(X,L)$, 
\begin{equation}\label{Kostantmufla}
\mu\,s
=\frac{\sqrt{-1}}{2\pi}\left(\nabla^L_{\xi}s-\frac{d}{dt}\Big|_{t=0}\varphi_t^* s\right)\,.
\end{equation}
\end{defi}
Formula \cref{Kostantmufla} is called the \emph{Kostant formula}.
The right-hand side of formula
\cref{Kostantmufla} is
$\cinf(X,\R)$-linear, so that the left-hand side defines in fact a function
$\mu\in\cinf(X,\R)$, invariant by the action of $S^1$.
Furthermore, \cref{Kostantmudef} together with the prequantization formula
\cref{preq} implies that for all $\eta\in\cinf(X,TX)$, we have
\begin{equation}\label{momentfla}
\om(\eta,\xi)=d \mu.\eta\,,
\end{equation}
where $\xi\in\cinf(X,TX)$ is the fundamental vector field \cref{xidef},
making the action of $S^1$ into a
\emph{Hamiltonian action} on the symplectic manifold $(X,\om)$,
with Hamiltonian $\mu\in\cinf(X,\R)$.

Now for any $p\in\N$, the Kostant formula \cref{Kostantmufla} applied
to any $s\in\cinf(X,L^p)$ gives
\begin{equation}\label{KostantK}
\frac{d}{dt}\Big|_{t=0}\varphi_{t}^* s=
\big(\nabla^{L^p}_{\xi}+2\pi\sqrt{-1}p\mu\big)s\,.
\end{equation}
For all $t\in\R$ and $x\in X$,
write $\tau_{t,p}:L^p_{\varphi_t(x)}\to L^p_x$
for the parallel transport with respect to
$\nabla^{L^p}$
along the path $u\mapsto\varphi_u(x)$ for $u\in[0,t]$ from $u=t$ to $u=0$,
defined for any $s\in\cinf(X,L^p)$ by
\begin{equation}\label{tautp}
\left(\nabla^{L^p}_\xi s\right)(x)=\frac{d}{dt}\Big|_{t=0}\tau_{t,p}\,s(\varphi_t(x))\,.
\end{equation}
Then formulas \cref{action} and
\cref{KostantK} imply that for any $p\in\N$ and $t\in\R$, we have
\begin{equation}\label{alphaf}
\varphi_{t,p}^{-1}=e^{2\pi\i tp \mu}\,\tau_{t,p}\,.
\end{equation}
In the sequel, we will always assume that the action of $S^1$ on $X$ preserves an
integrable
complex structure $J\in\cinf(X,\End(TX))$ compatible with $\om$ as in \cref{setting},
so that we have an induced action of $S^1$ on the space $H^0_{(2)}(X,L^p)$ of
$L^2$-holomorphic sections preserving the $L^2$-product
\eqref{L2}, and
for each $m\in\Z$, we define the
weight space $H^0_{(2)}(X,L^p)_m\subset H^0_{(2)}(X,L^p)$
of weight $m\in\Z$
as in \cref{wghtspace} and the Hilbert direct sum
$\HH_p\subset H^0_{(2)}(X,L^p)$ of the
negative weight spaces as in \cref{Hpdef}.

\section{Partial and equivariant Bergman kernels}
\label{parteqBergsec}

In this section,
we consider a Kähler manifold $(X,\om,J)$ prequantized by
$(L,h^L,\nabla^L)$ in the sense of \cref{preq}
with bounded geometry
at infinity in the sense of \cref{setting} and endowed
with a compatible $S^1$-action as in \cref{actionsec}.

In \cref{eqBergsec}, we introduce the equivariant Bergman
kernels and establish their off-diagonal properties, while in \cref{coordsec},
we establish their full off-diagonal expansion as $p\to\infty$.
In \cref{partBergsec}, we introduce the partial Bergman kernel
and use the results of \cref{eqBergsec,coordsec} to establish its
full off-diagonal expansion as $p\to\infty$.


\subsection{Equivariant Bergman kernels}
\label{eqBergsec}

Recall the Bergman kernel defined
in \cref{Bergprop} for any $p\in\N$
as the smooth Schwartz kernel of the orthogonal projection
$P_p:L^2(X,L^p)\to H^0_{(2)}(X,L^p)$.
We start this section by the following Proposition,
which serves as a definition of the equivariant
Bergman kernels. 




\begin{prop}\label{Pjprop}
For any $m\in\Z$, $p\in\N$,
the section
$P_p^{(m)}(\cdot,\cdot)
\in\cinf(X\times X,L^p\boxtimes (L^p)^*)$
defined for all $x,\,y\in X$
by the formula
\begin{equation}\label{Pjfla}
P_p^{(m)}(x,y):=\int_0^1\,e^{-2\pi\sqrt{-1} tm}
\varphi_{t,p}^{-1}P_p(\varphi_t(x),y)\,dt\,,
\end{equation}
coincides with the smooth Schwartz kernel with respect to $dv_X$ of
the orthogonal projection
$P_p^{(m)}:L^2(X,L^p)\to H^0_{(2)}(X,L^p)_m$
with respect to the $L^2$-product \cref{L2}
onto the weight space \cref{wghtspace} of weight $m\in\Z$.

Furthermore, for all $t\in\R$, $p\in\N$
and $x,\,y\in X$, it satisfies
\begin{equation}\label{Pmeq}
\varphi_{t,p}^{-1}P_p^{(m)}(\varphi_t(x),y)=
P_p^{(m)}(x,\varphi_t^{-1}(y))\varphi_{t,p}^{-1}=e^{2\pi\sqrt{-1} tm}
P_p^{(m)}(x,y).
\end{equation}
\end{prop}
\begin{proof}
For any $p\in\N$, $s\in H^0_{(2)}(X,L^p)$, $m\in\Z$
and $u\in\R$, using $\varphi_u\varphi_t=\varphi_{t+u}$
and the change of variable
$t\mapsto t-u$ for all $t\in S^1\simeq\R/\Z$, we compute
\begin{equation}\label{Pjflacomput}
\begin{split}
\varphi_{u}^*\left(\int_0^1\,e^{-2\pi\sqrt{-1} tm}
\varphi_{t}^*s\,dt\right)&=\int_0^1\,e^{-2\pi\sqrt{-1} tm}
\varphi_{t+u}^*s\,dt\\
&=e^{2\pi\sqrt{-1} um}\int_0^1\,e^{-2\pi\sqrt{-1} tm}
\varphi_{t}^*s\,dt\,.
\end{split}
\end{equation}
Together with the definition \cref{wghtspace} of the weight space
of weight $m\in\Z$, since the action of
$S^1$ on $L^2(X,L^p)$ is unitary
and recalling that
$P_p:L^2(X,L^p)\to H^0_{(2)}(X,L^p)$ denotes the orthogonal
projection, this shows that the map
\begin{equation}\label{Pmdef}
\begin{split}
P_p^{(m)}:L^2(X,L^p)&\longrightarrow H^0_{(2)}(X,L^p)_m\\
s&\longmapsto\int_0^1\,e^{-2\pi\sqrt{-1} tm}
\varphi_{t}^* P_ps\,dt
\end{split}
\end{equation}
is the orthogonal projection.
This implies formula \cref{Pjfla}
by \cref{Bergprop} introducing the Bergman kernel 
and definition \cref{action} of the action of $S^1$ on
sections, while formula \cref{Pmeq}
for the equivariance follows from the computation \cref{Pjflacomput}
in the same way,
together with the usual formula $P_p^{(m)}(x,y)=P_p^{(m)}(y,x)^*$
for all $x,\,y\in X$, holding
for Schwartz kernels of smoothing Hermitian operators.
\end{proof}


%


The Schwartz kernel defined in \cref{Pjprop} for all $p\in\N$
is called the \emph{equivariant Bergman kernel}
of weight $m\in\Z$. Recalling the fundamental vector field \cref{xidef} of the $S^1$-action,
the following result is adapted from \cite[Prop.\,4.1]{Ioo19}.

\begin{prop}\label{oscvener}
For any $k\in\N$ and $r\in\N$, there exists $C_k>0$ such that
for any $m\in\N$, $p\in\N$ and
$x,\,y\in X$ with $\mu(x)\neq m/p$,
we have
\begin{equation}\label{estoscvener}
\left|P_p^{(m)}(x,y)\right|_{\CC^r}
\leq \frac{C_k\,|\xi_x|^k}{\big|\mu(x)-\frac{m}{p}\big|^k}\,
p^{n+r-\frac{k}{2}}\,.
\end{equation}
\end{prop}
\begin{proof}
For any $x\in X$ and $t_0\in\R$,
consider a chart around $x_0:=\varphi_{t_0}(x)\in X$
as in \cref{chart} such that the radial line
in $B^{T_{x_0}X}(0,\epsilon_0)$ generated by
the fundamental vector field $\xi_{x}\in T_xX$ defined
by \cref{xidef} is sent to the
path $u\mapsto\varphi_u(x_0)$ in the image $U_x\subset X$
of $\psi_x:B^{T_xX}(x,\epsilon_0)\to X$.
Then $L^p$ is identified with $L^p_{x_0}$ along this path
by the parallel transport $\tau_{t,p}$ introduced in \cref{tautp},
for all $p\in\N$ and
$|t|<\epsilon_0$. Thus for any $Z\in B^{T_{x_0}X}(0,\epsilon)$
and any $t\in\R$ small enough, in the notation \eqref{Kxdef}
we have
\begin{equation}
\tau_{t,p}P_{p}(\varphi_{t_0+t}(x),\psi_{x_0}(Z))=
P_{p,x_0}(t\xi_{x_0},Z)\,.
\end{equation}
Using also that $\tau_{t_0+t,p}=\tau_{t,p}\tau_{t_0,p}$,
we can thus apply
\cref{theta,asy} for all $t_0\in S^1\simeq\R/\Z$,
so that for any $r,\,k\in\N$, we get $C_k>0$
such that for any
$x,\,y\in X,\,t\in[0,1]$ and $p\in\N$, we have
\begin{equation}\label{majdtTau}
\left|\Dk{k}{t}\tau_{t,p}P_{p}(\varphi_t(x),y)
\right|_{\CC^r} \leq C_k\,p^{n+r+\frac{k}{2}}\,|\xi_x|^k\,.
\end{equation}
Then using the exponentiated form \cref{alphaf} of the Kostant formula,
we can integrate by parts
to get from \cref{Pjprop} that
for any $x,y\in X$ and $k\in\N$,
\begin{multline}\label{oscvenerfla}
P^{(m)}_p(x,y)=\int_0^1 \tau_{t,p}
P_{p}(\varphi_t(x),y) e^{2\pi\sqrt{-1} t(p\mu(x)-m)} dt\\
=\frac{1}{(2\pi\sqrt{-1} (p\mu(x)-m))^k}\int_0^1
\tau_{t,p}P_{p}(\varphi_t(x),y) \Dk{k}{t}
 e^{2\pi\sqrt{-1} t(p\mu(x)-m)} dt\\
=\left(\frac{\sqrt{-1}}{2\pi}\right)^k
\frac{p^{-k}}{(\mu(x)-\frac{m}{p})^k}\int_0^1 \Dk{k}{t}
\left(\tau_{t,p}P_{p}(\varphi_t(x),y)
\right)e^{2\pi\sqrt{-1} t(p\mu(x)-m)} dt\,.
\end{multline}
Together with \cref{majdtTau}, this establishes the result.
\end{proof}

\subsection{Asymptotic expansion of equivariant Bergman kernels}
\label{coordsec}

Recall the setting described at the beginning of this section, and let us
assume in addition
that $S^1$ acts freely on $\mu^{-1}(0)\subset X$.
This means in particular that the fundamental
vector field $\xi\in\cinf(X,TX)$ defined by \cref{xidef} nowhere vanishes over
$\mu^{-1}(0)\subset X$, and the definition \cref{gTX} of the
Kähler metric together with the moment map equation \cref{momentfla}
then imply
that for all $x\in\mu^{-1}(0)$, we have
\begin{equation}\label{dmuJxi=1}
d\mu.J\xi_x=-|\xi_x|^2\neq 0\,.
\end{equation}
This shows in particular that $0$ is a regular value
of $\mu\in\cinf(X,\R)$, so that $\mu^{-1}(0)\subset X$
is a smooth submanifold.
For any $x\in\mu^{-1}(0)$, we set
\begin{equation}\label{N}
N_x:=\{\eta\in T_xX~|~g^{TX}(\eta,\xi_x)=g^{TX}(\eta,J\xi_x)=0\}\,.
\end{equation}
Note from \cref{gTX} and \cref{momentfla} that
$J\xi_x\in T_xX$
is orthogonal to $T_x\mu^{-1}(0)\subset T_xX$, and
consider the unit orthogonal vectors in $T_xX$ defined by
\begin{equation}\label{e_1e_2def}
e_0:=\frac{\xi_x}{|\xi_x|}\in T_x\mu^{-1}(0)\quad\text{ and }
\quad e_1:=\frac{J\xi_x}{|\xi_x|}\in T_x\mu^{-1}(0)^{\perp}\,,
\end{equation}
so that we have an orthogonal decomposition of $(T_xX,g^{TX}_x)$ given by
\begin{equation}\label{THXdec}
T_xX=\R\,e_0\oplus\R\,e_1\oplus N_x\,.
\end{equation}
Fix now $x_0\in\mu^{-1}(0)$.
In order to build a chart around $x_0\in X$
adapted to this decomposition,
let $\epsilon_0>0$ be such that
the exponential map
$\exp^{\mu^{-1}(0)}_{x_0}:
B^{T_{x_0}\mu^{-1}(0)}(0,\epsilon_0)\to \mu^{-1}(0)$
of $(\mu^{-1}(0),g^{TX}|_{\mu^{-1}(0)})$
is injective, and write 
\begin{equation}\label{BNdef}
V_0:=\exp^{\mu^{-1}(0)}_{x_0}(B^{N_{x_0}}(0,\epsilon_0))
\subset \mu^{-1}(0)\,,
\end{equation}
where $N_{x_0}\subset T_{x_0}\mu^{-1}(0)$ by construction.
This defines a local section of the
quotient map
$\pi:\mu^{-1}(0)\to X_0:=\mu^{-1}(0)/S^1$
considered as an $S^1$-principal bundle.
Let now $\epsilon_0>0$ be so small that
for all $u\in I_0:=\,]-\epsilon_0,\epsilon_0[$ and
$x\in V_0$, the ODE
\begin{equation}\label{phidef}
\frac{d}{du}\phi_u(x)=-\frac{J\xi_{\phi_u(x)}}{|\xi_{\phi_u(x)}|^2}
\quad\text{with}\quad\phi_0(x)=x
\end{equation}
defines an embedding
$\phi:I_0\times V_0\xrightarrow{\sim} V\subset X$, which
by \cref{dmuJxi=1} satisfies  $\mu(\phi_u(x))=u$ for all
$u\in I_0$ and $x\in V_0$. Since
formula \cref{phidef} is $S^1$-invariant in the sense
that for all $t\in S_1\simeq\R/Z$, $u\in\in I_0$ and
$x\in V_0$, we have
$\phi_u(\varphi_t(x))=\varphi_t(\phi_u(x))$,
we get a natural $S^1$-equivariant open
embedding
\begin{equation}\label{eqemb}
\begin{split}
\Phi:S^1\times I_0\times V_0&\xrightarrow{~\sim~}U\subset X\\
(t\,,\,u\,,\,x)&\longmapsto\varphi_t(\phi_u(x))\,.
\end{split}
\end{equation}
Using a partition of unity, we can then consider a
bounded family of charts
$\psi:B^{TX}(0,\epsilon_0)\to X$ in the sense of
\cref{chart} such that for any $x_0\in\mu^{-1}(0)$,
its restriction
$\psi_{x_0}:B^{T_{x_0}X}(0,\epsilon_0)\xrightarrow{\sim} U_{x_0}\subset X$
is defined for any $t,\,u\in\R$ and $\eta\in N_{x_0}$
satisfying $|te_1+ue_1+\eta|<\epsilon_0$
in the decomposition \cref{THXdec} by
\begin{equation}\label{psidef}
\psi_{x_0}(t e_0+u e_1+\eta)=
\Phi\left(\frac{t}{|\xi_{x_0}|}\,,\,-|\xi_{x_0}|\,u\,,\,
\exp^{\mu^{-1}(0)}_{x_0}(\eta)\right)\,.
\end{equation}
Note that we have in fact $d\psi_{x_0}=\Id_{T_{x_0}X}$
by \cref{e_1e_2def} and \cref{phidef}.
For any $x\in\mu^{-1}(0)$, we consider the \emph{horizontal tangent bundle}
\begin{equation}\label{THX}
\begin{split}
T_x^HX&:=\{\eta\in T_xX~|~g^{TX}(\eta,\xi_x)=0\}\\
&=\R e_1\oplus N_x\subset T_x X
\end{split}
\end{equation}
and for any $Z\in T_x^HX$,
%
we write 
\begin{equation}\label{Zperp}
Z=u\,e_1+Z^\perp\in\R\,e_1\oplus N_x\,,
\end{equation}
for its decomposition induced by \cref{THXdec},
with $u\in\R$ and $Z^\perp \in N_x$.
In the same way, for any $Z'\in T_x^HX$, we write 
$Z'=u'\,e_1+{Z'}^\perp$ with $u'\in\R$ and ${Z'}^\perp \in N_x$.
The following result establishes the
asymptotic expansion for the
equivariant Bergman kernels introduced in \cref{Pjprop}.

\begin{theorem}\label{wghtedasyth}
There is a family
$\{Q_{r,x}\in\C[Y,Z,Z']\}_{r\in\N}$ of
polynomials in $Z,Z'\in T_x^HX$ and $Y\in\R$,
depending smoothly on $x\in\mu^{-1}(0)$, of the same total
parity as $r\in\N$ and of total degree less than $3r\in\N$,
such that for any compact set $K\subset\mu^{-1}(0)$,
any $k,\,j\in\N$ and
$\delta>0$, there is $\epsilon_0>0$ and $C>0$
such that for all $\epsilon\in\,]0,\epsilon_0[$,
$x\in K$, $p\in\N$, $m\in\Z$
and $Z,\,Z'\in T_x^HX$ satisfying $|Z|,\,|Z'|<\epsilon\,p^{\frac{-1+\epsilon}{2}}$,
we have
\begin{equation}\label{wghtedasyfla}
\left|p^{-n+\frac{1}{2}}P_{p,x}^{(m)}(Z,Z')
-\sum_{r=0}^{k-1} Q_{r,x}\left(\frac{m}{\sqrt{p}},\sqrt{p}Z,\sqrt{p}Z'\right)\PP_x^{(m)}(\sqrt{p}Z,\sqrt{p}Z')p^{-\frac{r}{2}}\right|_{\CC^j(K)}\leq Cp^{-\frac{k}{2}+\delta},
\end{equation}
where for any $Z,\,Z'\in T_x^HX$ as in \cref{Zperp}, we have
\begin{multline}\label{wghtedlocmod}
\PP_x^{(m)}(Z,Z')
=
\frac{\sqrt{2}}{|\xi_x|}\exp\left(-2\pi\left(\frac{u+u'}{2}+\frac{m}{|\xi_x|\sqrt{p}}\right)^2\right)\\
\exp\left(-\frac{\pi}{2}(u-u')^2\right)\,\PP_x(Z^\perp,{Z'}^\perp)\,.
\end{multline}
Furthermore, for all $x\in M$,
we have $Q_{0,x}\equiv 1$.
%
\end{theorem}
\begin{proof}
Let $x\in\mu^{-1}(0)$, and consider the
trivialization of $L$ by parallel transport
along radial lines in the charts
defined by \cref{psidef}.
Then for any $Z\in B^{T^H_xX}(0,\epsilon_0)$
as in \cref{Zperp},
by a standard computation
which can be found for example in \cite[(1.2.31)]{MM07}
and which holds in any trivialization of $L$ along radial lines,
the connection $\nabla^L$ at
$\psi_x(Z)\in X$ has the form
\begin{equation}\label{nablatriv}
\begin{split}
\nabla^L_\xi &=d+\frac{1}{2}R^L(Z,\xi_x)+O(|Z|^2)\\
&=d-\sqrt{-1}\pi \om_x(Z,\xi_x)+O(|Z|^2)\\
&=d+\sqrt{-1}\pi |\xi_x|\,u+O(|Z|^2)\,,
\end{split}
\end{equation}
where we used formulas \cref{preq} and \cref{gTX} for $\om$ and $g^{TX}$ to
get the first line, and formulas \cref{e_1e_2def} 
and \cref{Zperp} for the second. On the other hand,
by \eqref{dmuJxi=1} and \cref{phidef}, 
we have $\mu(\psi_x(Z))=-|\xi_x|\,u$ in the charts
defined by \cref{psidef}.
Hence by definition \cref{tautp} of the parallel transport
with respect to $\nabla^{L^p}_\xi$ and by the form \cref{nablatriv}
of $\nabla^L$ at $Z\in B^{T^H_xX}(0,\epsilon_0)$, formula \cref{alphaf}
reads as
\begin{equation}\label{varphitpasy}
\begin{split}
\varphi_{t,p}^{-1}
&=\exp(2\pi\i t p \mu(\psi_x(Z)))\exp(\sqrt{-1}\pi tp (|\xi_x|\,u+O(|Z|^2)))\\
&=\exp(-\i\pi tp (|\xi_x|\,u+O(|Z|^2))\,.
\end{split}
\end{equation}
Let us now choose $\epsilon>0$ so small
that for any $Z\in B^{T_x^HX}(0,\epsilon)$ and
$t\in\,]-\epsilon,\epsilon[$, we have
$|te_0+Z|<\epsilon_0$. Then in the charts
defined by \cref{psidef}, we get
\begin{equation}\label{varphitZexp}
\varphi_t(Z)=|\xi_x|\,t\,e_0+Z\,.
\end{equation}
Hence by definition of the local model \cref{PPreal},
we have the following
asymptotic expansion as $p\to\infty$,
uniform in $t\in\,]-\eptt,\eptt[$ and in
$Z,\,Z'\in B^{T_x^HX}(0,\eptt)$ as in \cref{Zperp},
\begin{equation}\label{PPvarphitZZ'}
\begin{split}
&\PP_x(\sqrt{p}\,\varphi_{t/\sqrt{p}}(Z/\sqrt{p}),Z')\\
&=
\exp\left(-\frac{\pi}{2}|\xi_x|^2t^2-\frac{\pi}{2}|Z-Z'|^2
-\pi\i t \om_x(\xi_x,Z')-\pi\i t \om_x(Z,Z')\right)\exp\left(O(p^{\epsilon-\frac{1}{2}})\right)\\
&=\exp\left(-\frac{\pi}{2}|\xi_x|^2t^2-\pi\i |\xi_x| t u'\right)
\exp(-\frac{\pi}{2}(u-u')^2)\PP_x(Z^\perp,{Z'}^\perp)+O(p^{\epsilon-\frac{1}{2}})\,.
\end{split}
\end{equation}
Then by \cref{theta,thetagal,Pjprop,asy},
using the classical formula for the Fourier transform
of the Gaussian and the change of variables
$(t,Z,Z')\mapsto( t/\sqrt{p},Z/\sqrt{p},Z'/\sqrt{p})$,
for any
$\delta>0$, we get $\epsilon>0$ such that
we have the following asymptotic expansion as $p\to\infty$, uniform in $m\in\Z$, in $\epsilon\in\,]0,\epsilon_0[$ and
in $Z,\,Z'\in B^{T_x^HX}(0,\ept)$ as in \cref{Zperp},
\begin{equation}\label{Pjcomput}
\begin{split}
&p^{-n+\frac{1}{2}}\,P_{p,x}^{(m)}(Z/\sqrt{p},Z'/\sqrt{p})\\
&=p^{-n+\frac{1}{2}}\int_{-\ept}^{\ept}\,e^{-2\pi\sqrt{-1} tm}
\varphi_{t,p}^{-1}P_{p,x}(\varphi_t(Z/\sqrt{p}),Z'/\sqrt{p})\,dt
+O(p^{-\infty})\\
&=\int_{-\eptt}^{\eptt}\,e^{-2\pi\sqrt{-1} t\frac{m}{\sqrt{p}}}
\exp(-\i\pi t |\xi_x|\,u)\PP_x(\sqrt{p}\,\varphi_{t/\sqrt{p}}(Z/\sqrt{p}),Z')\,dt+O(p^{\delta-\frac{1}{2}})\\
&=\left(\int_{-\infty}^{\infty}\,e^{-\i\pi t|\xi_x|\,(u+u')-2\i\pi t\frac{m}{\sqrt{p}}-\frac{\pi}{2}|\xi_x|^2t^2}\,dt\right)
\exp(-\frac{\pi}{2}(u-u')^2)\PP_x(Z^\perp,{Z'}^\perp)+O(p^{\delta-\frac{1}{2}})\\
&=\frac{\sqrt{2}}{|\xi_x|}\exp\left(-2\pi\left(\frac{(u+u')}{2}+\frac{m}{|\xi_x|\sqrt{p}}\right)^2\right)
\exp(-\frac{\pi}{2}(u-u')^2)\PP_x(Z^\perp,{Z'}^\perp)+O(p^{\delta-\frac{1}{2}})\,.
\end{split}
\end{equation}
This establishes the first order of the
asymptotics expansion \cref{wghtedasyfla}. To get the next order, one
simply needs to consider the next order of the Taylor expansions
performed in formulas \cref{nablatriv,varphitpasy,varphitZexp,PPvarphitZZ'},
then realize that all integrations performed in the computation \cref{Pjcomput}
extend to the case when the exponential
terms are multiplied by polynomials
in $(\sqrt{p}t,\sqrt{p}Z,\sqrt{p}Z')$, leading to an asymptotic expansion of
the form \cref{wghtedasyfla}.
This proves the result.
\end{proof}

\begin{rmk}
\cref{wghtedasyth} gives the asymptotic expansion as $p\to\infty$
of the equivariant Bergman kernel with respect to each
weight $m\in\Z$ in a neighborhood
of $\mu^{-1}(0)$ and in
the directions orthogonal to the $S^1$-action, since it only holds for
$Z,\,Z'\in T^{H}X$ as in \cref{THX}.
Together with \cref{Pjprop} prescribing its
behavior in the direction of the $S^1$-action and \cref{oscvener}
giving its asymptotic decay outside the same neighborhood of $\mu^{-1}(0)$,
\cref{wghtedasyth} thus computes the full asymptotic expansion of the
equivariant Bergman kernels as $p\to\infty$, uniformly in
the weight $m\in\Z$.
In the special case $m=0$,
\cref{wghtedasyth} is a consequence of the
result of Ma and Zhang in \cite[Th.\,0.2]{MZ08}.
\end{rmk}

\subsection{Asymptotic expansion of the partial Bergman kernel}
\label{partBergsec}

Recall the setting described at the beginning of the section.
The following Proposition serves as a definition
of the partial Bergman kernel, and gives its first
fundamental asymptotic properties.

\begin{theorem}\label{faroffdiagprop}
For any $p\in\N$, the orthogonal
projection to the space $\HH_p\subset L^2(X,L^p)$
defined by \cref{Hpdef} admits a smooth Schwartz kernel
$P_p^{(-)}(\cdot,\cdot)\in\cinf(X\times X,L^p \boxtimes (L^p)^*)$, and
%
%
for any $r,\,k\in\N$ and $\alpha>0$,
there is $C_k>0$ such that for all
$\epsilon\in\,]0,\alpha/2[$, $p\in\N$ and $x,\,y\in X$,
the following estimates holds
\begin{equation}\label{P-sum1}
\begin{split}
\left|P_{p}^{(-)}(x,y)-\sum_{m\geq-\alpha p^{\frac{1+\alpha}{2}}}^{m=0}
\,P_p^{(m)}(x,y)\right|_{\CC^r}\leq C_k\,p^{-k}\,|\xi_x|^k &
\quad\text{if}\quad|\mu(x)|\leq\epsilon p^{\frac{-1+\epsilon}{2}}\,,\\
|P_p^{(-)}(x,y)|_{\CC^r}\leq C_k\,p^{-k}\,|\xi_x|^k &
\quad\text{if}\quad\mu(x)>\alpha p^{\frac{-1+\alpha}{2}}\,,\\
|P_p^{(-)}(x,y)-P_p(x,y)|_{\CC^r}\leq C_k\,p^{-k}\,|\xi_x|^k &
\quad\text{if}\quad\mu(x)<-\alpha p^{\frac{-1+\alpha}{2}}\,,
\end{split}
\end{equation}
and $|P_p^{(-)}(x,y)|_{\CC^r}\leq C_k\,p^{-k}\,|\xi_x|^k$
if $d^X(x,y)>\alpha p^{\frac{-1+\alpha}{2}}$.
\end{theorem}
\begin{proof}
For any $p\in\N$ and $m\in\Z$,
recall the orthogonal projection \cref{Pmdef}
on the weight space \cref{wghtspace} of weight
$m\in\Z$, and note from \cref{Hpdef}
that the map
\begin{equation}\label{P-=sumPm}
P_p^{(-)}:=\sum_{m\leq 0} P_p^{(m)}:L^2(X,L^p)
\longrightarrow\HH_p
\end{equation}
converges to the orthogonal projection onto
the space \cref{Hpdef} in the space
of bounded operators acting on $L^2(X,L^p)$.
In the same way, by the weight decomposition \cref{decwghtspace},
the orthogonal
projection $P_p:L^2(X,L^p)\to H^0_{(2)}(X,L^p)$
satisfies $P_{p}=\sum_{m\in\Z}\,P_p^{(m)}$, where the
sum converges not only in the space
of bounded operators acting on $L^2(X,L^p)$, but
also in local $\CC^r$-norms
of their Schwartz kernels for each $r\in\N$
by \cref{Pjprop} and standard Fourier theory.
This shows in particular that the sum \eqref{P-=sumPm}
converges in local $\CC^r$-norms for each $r\in\N$,
so that \eqref{P-=sumPm} admits a smooth
Schwartz kernel
$P_p^{(-)}(\cdot,\cdot)\in\cinf(X\times X,L^p \boxtimes (L^p)^*)$,
and we have the following smooth convergence
for all $x,\,y\in X$,
\begin{equation}\label{kerP-=sumPm}
P_{p}^{(-)}(x,y)=\sum_{m\in\Z}\,P_p^{(m)}(x,y)\,.
\end{equation}
Using \cref{oscvener}, for any
$k,\,r\in\N$ and $\alpha>0$, we get
a constant $C_k>0$ such that for any
$\epsilon\in\,]0,\alpha/2[$,
$p\in\N$ and $x,\,y\in X$ such that
$|\mu(x)|\leq\epsilon p^{\frac{-1+\epsilon}{2}}$,
we have
\begin{equation}\label{faroffdiagcomput}
\begin{split}
\sum_{m<-\alpha p^{\frac{1+\alpha}{2}}}
\left|P_p^{(m)}(x,y)\right|_{\CC^r}&\leq
C_k\,p^{n+r-\frac{k}{2}}\,|\xi_x|^k
\sum_{m< -\alpha p^{\frac{1+\alpha}{2}}}\,\left|\mu(x)-\frac{m}{p}\,\right|^{-k}\\
&\leq
C_k p^{n+r-\frac{k}{2}}\,|\xi_x|^k
\sum_{m=1}^{\infty}\,\left|\big(\mu(x)-2\alpha p^{\frac{-1+\alpha}{2}}\big)+\frac{m}{p}\,\right|^{-k}\\
&\leq
C_k \alpha^{-k} p^{-\alpha\frac{k}{2}}
p^{n+r}\,|\xi_x|^k
\sum_{m=1}^\infty\,\left(\frac{m}{\epsilon p^{\frac{1+\alpha}{2}}}+1\right)^{-k}\\
&\leq
C_k \alpha^{1-k}
p^{n+r+\frac{1}{2}}p^{-\alpha\frac{k-1}{2}}\,|\xi_x|^k
\int_{0}^{\infty}\,(t+1)^{-k}\,dt\,.
\end{split}
\end{equation}
Since the integral in the last line of \cref{faroffdiagcomput}
converges,
this shows that the sum in the left-hand side converges as
well and proves the first estimate of \cref{P-sum1} by
formula \eqref{kerP-=sumPm} and taking $k\in\N$ large enough.


Letting now $x,\,y\in X$ satisfy
$\mu(x)>\alpha p^{\frac{-1+\alpha}{2}}$,
a strictly analogous computation
allows to directly estimate the sum
$\sum_{m\leq 0} |P_p^{(m)}(x,y)|_{\CC^r}$ for all $r\in\N$ in the
same way, establishing the second estimate of \cref{P-sum1}
by formula \eqref{kerP-=sumPm}.

%

Finally, if $x,\,y\in X$ satisfy
$\mu(x)<-\alpha p^{\frac{-1+\alpha}{2}}$,
a computation stricly analogous to \cref{faroffdiagcomput}
allows to estimate the sum
$\sum_{m=1}^{\infty} |P_p^{(m)}(x,y)|_{\CC^r}$
for all $r\in\N$ in the same way,
giving the third estimate of \cref{P-sum1}
thanks to the identity
$P_p-P_p^{(-)}=\sum_{m=1}^{\infty} P_p^{(m)}$ following
from \eqref{P-=sumPm} as above. This concludes the proof.

%
\end{proof}


The Schwartz kernel defined in \cref{faroffdiagprop}
for all $p\in\N$ is called the \emph{partial Bergman kernel}.
Recalling the bounded family of charts
constructed in \cref{coordsec},
the main goal of this section is to establish the
following result on the near-diagonal expansion of the
partial Bergman kernel.

\begin{theorem}\label{partBergasyth}
There are two families
$\{Q_{r,x}^{(-)}\in\C[Z,Z']\}_{r\in\N}$ 
and $\{Q_{r,x}^{(0)}\in\C[Z,Z']\}_{r\in\N}$ of
polynomials in $Z,\,Z'\in T_x^HX$ depending smoothly on
$x\in\mu^{-1}(0)$, of the same total
parity as $r\in\N$ and of degree less that $3r\in\N$,
such that for any compact set $K\subset\mu^{-1}(0)$,
any $k,\,j\in\N$ and
$\delta>0$, there is $\epsilon_0>0$ and $C>0$
such that for any $\epsilon\in\,]0,\epsilon_0[$,
$x\in K$, $p\in\N$
and
$Z,\,Z'\in T_x^HX$
satisfying $|Z|,\,|Z'|<\ept$,
we have
\begin{multline}\label{partBergasyfla}
\Big|p^{-n}P_{p,x}^{(-)}(Z,Z')
-\sum_{r=0}^{k-1} Q_{r,x}^{(-)}(\sqrt{p}Z,\sqrt{p}Z')\PP_x^{(-)}(\sqrt{p}Z,\sqrt{p}Z')p^{-\frac{r}{2}}\\
\left.- p^{-\frac{1}{2}}\,
\sum\limits_{r=0}^{k-2}Q_{r,x}^{(0)}(\sqrt{p}Z,\sqrt{p}Z')
\PP_x^{(0)}(\sqrt{p}Z,\sqrt{p}Z')p^{-\frac{r}{2}}
\right|_{\CC^{j}(K)}\leq Cp^{-\frac{k}{2}+\delta}.
\end{multline}
where $\PP_x^{(0)}$ is the local model \cref{wghtedlocmod}
for $m=0$ and where for any $Z,\,Z'\in T_x^HX$
 as in \cref{Zperp}, we have
\begin{equation}\label{partBerglocmod}
\PP_x^{(-)}(Z,Z')
=\sqrt{2}\left(\int_{-\infty}^{\frac{u+u'}{2}}
e^{-2\pi t^2}\,dt\right)
\exp\left(-\frac{\pi}{2}(u-u')^2\right)\,\PP_x(Z^\perp,{Z'}^\perp)\,.
\end{equation}
Furthermore, for all $x\in\mu^{-1}(0)$ and $Z,\,Z'\in T_x^HX$, we have
$Q^{(-)}_{0,x}\equiv 1$.

\end{theorem}
\begin{proof}
First recall the
following Euler-MacLaurin formula, which can be found for instance in
\cite[Th.\,9.1]{RS17}, and which states
the existence of a universal sequence
$\{a_j\in\R\}_{j\in\N}$ such that
for any $f\in\cinf(\R,\R)$
whose derivatives of all order
tend to $0$ at infinity, and for any $r\in\N$, we have
\begin{equation}\label{EM1}
\sum_{m\in\N}\,f(m)=\int_0^{\infty}\,f(t)\,dt+\sum_{j=0}^{r-1}\,a_j\,f^{(j)}(0)
+\int_0^{\infty}\,a_r\,(t-\left \lceil{t}\right \rceil )\,f^{(r)}(t)\,dt\,.
\end{equation}
In particular, this implies the
existence of polynomials $\{P_{a,j}\in\R[v]\}_{j\in\N}$
of degree at most $j\in\N$ and
depending smoothly on $a>0$
such that for any compact interval $I\subset\,]0,+\infty[$,
we have the following estimate as $p\to\infty$, uniform in $v\in\R$ and $a\in I$,
\begin{equation}\label{EM2}
\sum_{m\in\N}\,e^{-2\pi\big(v-\frac{m}{a\sqrt{p}}\big)^2}=a\sqrt{p}\int_{-\infty}^v\,e^{-2\pi t^2}\,dt
+\sum_{j=0}^{r-1}\,p^{-\frac{j}{2}}\,P_{a,j}(v)
\,e^{-2\pi v^2}
+O(p^{-\frac{r}{2}})\,.
\end{equation}
More generally,
for any $k\in\N$, we get sequences of
polynomials $\{P_{a,k,j},\,\widetilde{P}_{a,k,j}\in\R[v]\}_{j\in\N}$
of degree at most $j+k\in\N$ and
depending smoothly on $a>0$
such that for any compact interval $I\subset\,]0,+\infty[$,
we have the following estimate
as $p\to\infty$, uniform in $v\in\R$ and $a\in I$,
\begin{equation}\label{EM3}
\begin{split}
&\sum_{m\in\N}\,\left(\frac{m}{\sqrt{p}}\right)^k\,
e^{-2\pi\big(v-\frac{m}{a\sqrt{p}}\big)^2}\\
&=a^{k+1}\sqrt{p}\int_{-\infty}^v\,\left(v-t\right)^k\,e^{-2\pi t^2}\,dt+\sum_{j=0}^{r-1}\,p^{-\frac{j}{2}}\,P_{a,k,j}(v)\,e^{-av^2}
+O(p^{-\frac{r}{2}})\,.\\
&=a^{k+1}\sqrt{p}\left(v^k\int_{-\infty}^v\,e^{-2\pi t^2}\,dt
+\widetilde{P}_{a,k,0}(v)\,e^{-a v^2}\right)
+\sum_{j=0}^{r-1}\,p^{-\frac{j}{2}}\,\widetilde{P}_{a,k,j+1}(v)
\,e^{-2\pi v^2}
+O(p^{-\frac{r}{2}})\,.
\end{split}
\end{equation}
Then by \cref{wghtedasyth,faroffdiagprop}, for any compact
subset $K\subset\mu^{-1}(0)$, $k\in\N$ and $\delta>0$,
we can choose $\alpha>0$ such that we
have the following uniform estimate as $p\to\infty$,
uniform in $x\in K$, in $\epsilon\in\,]0,\alpha/2[$
and in $Z,\,Z'\in T^H_xX$ satisfying $|Z|,\,|Z'|<\ept$,
\begin{equation}\label{P-sum2}
\begin{split}
&p^{-n}P_{p,x}^{(-)}(Z,Z')
=p^{-n}\sum_{m\geq-\alpha p^{\frac{1+\alpha}{2}}}^{m=0}P_{p,x}^{(m)}(Z,Z')+O(p^{-\infty})\\
&=p^{-\frac{1}{2}}\sum_{r=0}^{k-1} p^{-\frac{r}{2}} \sum_{m\geq-\alpha p^{\frac{1+\alpha}{2}}}^{m=0}
Q_{r,x}\left(\frac{m}{\sqrt{p}},\sqrt{p}Z,\sqrt{p}Z'\right)
\PP_x^{(m)}(\sqrt{p}Z,\sqrt{p}Z') +
O(p^{-\frac{k}{2}+\delta})\,.
\end{split}
\end{equation}
On the other hand, for any $r\in\N$,
using the local model \cref{wghtedlocmod}
for the equivariant Bergman kernel
and recalling from \cref{wghtedasyth}
that the polynomial $Q_r$
in formula \cref{P-sum2}
is a polynomial of degree less
than $3r\in\N$ in $m/\sqrt{p}$,
for any $\alpha>0$ small enough,
we get constants $C_r>0$ and $c>0$ such that
for all $\epsilon\in\,]0,\alpha/2[$, $x\in K$ and $Z,\,Z'\in T^H_xX$ as in \cref{Zperp}
satisfying $|Z|,\,|Z'|<\eptt$,
we have
\begin{equation}\label{sumcomput1}
\begin{split}
&\sum_{m<-\alpha p^{\frac{1+\alpha}{2}}}
\left|Q_{r,x}\left(\frac{m}{\sqrt{p}},Z,Z'\right)
\PP_x^{(m)}(Z,Z')\right|\\
&\leq C\,p^{\frac{3r}{2}\alpha}\sum_{m<-\alpha p^{\frac{1+\alpha}{2}}}\,
\left|\frac{m}{\sqrt{p}}\right|^{3r}
\exp\left(-2\pi\left(\frac{u+u'}{2}-\frac{m}{|\xi_x|\sqrt{p}}\right)^2\right)\\
&\leq C_r\,p^{\frac{3r}{2}\alpha+\frac{1}{2}}\,
\exp\left(-c\,p^{\frac{\alpha}{2}}\right)\,,
\end{split}
\end{equation}
which is decreasing faster than any power of $p\in\N$.
In particular, we can apply \cref{EM3} with
$v=(u+u')/2$ and $a=|\xi_x|$
to get from \cref{P-sum2,sumcomput1}
sequences of polynomials
$\{P_r,\,R_{r,k}\in\C(Z,Z')\}_{r,\,k\in\N}$ such
that for any compact subset $K\subset\mu^{-1}(0)$ and
$\delta>0$, there is $\alpha>0$
such that we have the following asymptotic
expansion as $p\to\infty$,
for all $\epsilon\in\,]0,\alpha/2[$, $x\in K$
and $Z,\,Z'\in T^H_xX$ as in \cref{Zperp}
satisfying $|Z|,\,|Z'|<\eptt$,
\begin{equation}\label{sumcomput2}
\begin{split}
\sum_{m\geq-\alpha p^{\frac{1+\alpha}{2}}}^{m=0}
Q_{r,x}\left(\frac{m}{\sqrt{p}},Z,Z'\right)
\PP_x^{(m)}(Z,Z')&=
\sum_{m\leq 0}
Q_{r,x}\left(\frac{m}{\sqrt{p}},Z,Z'\right)
\PP_x^{(m)}(Z,Z')+O(p^{-\infty})\\
=\frac{\sqrt{2}}{|\xi_x|}\sum_{m\in\N}Q_{r,x}\left(-\frac{m}{\sqrt{p}},Z,Z'\right)
\exp&\left(-2\pi\left(\frac{(u+u')}{2}+\frac{m}{|\xi_x|\sqrt{p}}\right)^2\right)\\
&\exp\left(-\frac{\pi}{2}(u-u')^2\right)\,\PP(Z^\perp,{Z'}^\perp)+O(p^{-\infty})\\
=\sqrt{p}\,P_{r}(Z,Z')
\PP^{(-)}(Z,Z')+\sqrt{p}&\,R_{r,0}(Z,Z')\PP^{(0)}(Z,Z')\\
+\sum_{k=0}^{j-1}p^{-\frac{k}{2}}\,
&R_{r,k}(Z,Z')\PP^{(0)}(Z,Z')+O(p^{-\frac{j}{2}+\delta})\,,
\end{split}
\end{equation}
where $P_0\equiv 1$ and $R_{0,0}\equiv 0$ by equation \cref{EM2}
and the fact that $Q_0\equiv 1$ in the asymptotic
expansion of \cref{wghtedasyth}. Plugging \cref{sumcomput2}
into \cref{P-sum2} after the change of variables
$(Z,Z')\mapsto (\sqrt{p}Z,\sqrt{p}Z')$, we then get the result.
\end{proof}

\begin{rmk}\label{RSrmk}
For any $x\in\mu^{-1}(0)$, by
definition \cref{psidef}
of the bounded family of charts used in \cref{wghtedasyth},
we know that there exists $\epsilon>0$
such that for any $u\in\,]-\epsilon,\epsilon[$,
we have  $\phi_u(x)=\psi_x(-\frac{u}{|\xi_x|}\,e_1)$,
where $\phi_u:\mu^{-1}(0)\to X$ is the flow of
$-J\xi/|\xi|^2$ at time $u\in\,]-\epsilon,\epsilon[$
as defined in \cref{phidef}, which satisfies
$\mu(\phi_u(x))=u$. Then
\cref{faroffdiagprop,partBergasyth} applied to
$Z=Z'=-\frac{u}{|\xi_x|}\,e_1$
recover
the asymptotic expansion of the partial Bergman on the diagonal
established by Ross and Singer
in \cite[Th.\,1.2]{RS17} and Zelditch
and Zhou in \cite[(8),\,Th.\,4]{ZZ19},
while the general result recovers
a result of Shabtai in \cite[Th.\,1.7]{Sha25}.
In all these results, the asymptotic expansions hold
in neighborhoods of size of order $\frac{1}{\sqrt{p}}$
around the boundary
$\mu^{-1}(0)\subset X$ and outside neighborhoods
of size of order $1$ as $p\to\infty$,
while \cref{faroffdiagprop,partBergasyth} give a full
asymptotic expansion in any neighborhood of $\mu^{-1}(0)$.
\end{rmk}

\section{Determinantal point processes}
\label{DPPsec}

In \cref{Gensec}, we introduce the notion of a determinantal
point process in our setting and recall its fundamental
properties following Berman in \cite{Ber18}, then
explain how it recovers the usual Ginibre ensemble
\cref{Gindef} in the case of $X=\C$ endowed with
the action of $S^1\subset\C$ by multiplication.
In \cref{LLNsec}, we establish a law of large numbers
for this determinantal point process, while in \cref{CLTsec},
we establish a central limit theorem, thus concluding
the proof of \cref{mainth}.

\subsection{Generalities}
\label{Gensec}

Let $(X,\om,J)$ be
a Kähler manifold prequantized by
$(L,h^L,\nabla^L)$ in the sense of \cref{preq},
with bounded geometry
at infinity in the sense of \cref{setting} and endowed
with a compatible $S^1$-action as in \cref{actionsec}.
We assume in addition that $S^1$ acts freely on
the level set $\mu^{-1}(0)\subset X$ of its Kostant
moment map $\mu\in\cinf(X,\R)$ introduced in
\cref{Kostantmudef}, and that it satisfies the following.

\begin{defi}\label{volgrowth}
We say that $\mu\in\cinf(X,\R)$
has \emph{polynomial growth} if it is proper,
bounded from below and
and if there exists $\delta>0$, $C>0$ and $N\in\N$ such that
for any $x\in X$
\begin{equation}
\Vol(\{\mu\leq \mu(x)\})\leq C\,|\mu(x)|^N
\quad\text{ and }\quad|d\mu_x|
\leq C\,|\mu(x)|^{1-\delta}\,.
\end{equation}
\end{defi}

Note that in the case $X=\C^n$, any positive polynomial in
the coordinates of $\C^n$ has polynomial growth in the sense of
\cref{volgrowth}.
We can
then establish the following result from \cref{oscvener},
which can be understood as a weak version of the
principle of \emph{Quantization commutes with Reduction} for 
non-compact manifolds established by
Ma and Zhang in \cite{MZ14}.

\begin{prop}\label{QR=0}
Assume that the Kostant moment map $\mu\in\cinf(X,\R)$ has
polynomial growth in the sense of \cref{volgrowth}.
Then there exists $p_0\in\N$ and $M\in\N$ such that for any $p\geq p_0$ and any $m\in\Z$ satisfying $m< -p\,M$,
we have
\begin{equation}\label{QR=0fla}
H^0_{(2)}(X,L^p)_m=\{0\}\,.
\end{equation}
In particular, the subspace
$\HH_p\subset H^0_{(2)}(X,L^p)$ defined by \cref{Hpdef}
satisfies
\begin{equation}\label{Hpfin}
\HH_p=\bigoplus\limits_{m=-M\,p}^{m=0}\,H^0_{(2)}(X,L^p)_m\,,
\end{equation}
and has finite dimension $N_p:=\dim\HH_p\in\N$.
\end{prop}
\begin{proof}
Assume that the moment map $\mu\in\cinf(X,\R)$ has
polynomial growth in the sense of \cref{volgrowth}.
In particular, it is bounded from below, and we can pick
$M\in\N$ such that for all $x\in X$, we have $\mu(x)>-M+1$. 
Note also from \cref{gTX} and \cref{momentfla}
that the fundamental vector field \cref{xidef}
satisfies $|\xi_x|=|d\mu_x|$.
Then \cref{volgrowth} and
\cref{oscvener} imply the existence of constants
$\delta,\,C>0$ such that
for any $k\in\N$,  there is a constant $C_k>0$ such that
for any $p\in\N$, any $m\in\N$ satisfying $m<-p\,M$
and any $q\in\Z$ satisfying $q\geq-M+1$, we have
\begin{equation}\label{Npcomput0}
\begin{split}
\int_{\{q\leq \mu<q+1\}}\,|P_p^{(m)}(x,x)|_p\,dv_X(x)
&\leq C_k\,p^{n-\frac{k}{2}}\,\int_{\{q\leq \mu<q+1\}}\,\left|\mu+M\right|^{-k}\,|d\mu|^k\,dv_X
\\
&\leq C_k\,C\,p^{n-\frac{k}{2}}
|q+1|^{N} |q+M|^{-k}|q+1|^{k(1-\delta)}\,.
\end{split}
\end{equation}
Taking $k>(N+2)/\delta$ and
summing over all $q\in\Z$ satisfying $q\geq-M+1$, by standard properties
of smooth Schwartz kernels of the orthogonal projection
$P^{(m)}_p:L^2(X,L^p)\to H^0_{(2)}(X,L^p)_m$, for any $k\in\N$,
we get a constant $C_k>0$ such that for all
$p\in\N$ and all $m\in\Z$ satisfying $m<-p\,M$, we have
\begin{equation}\label{dimHm}
\begin{split}
\dim H^0_{(2)}(X,L^p)_m&=\int_X\,P_p^{(m)}(x,x)\,dv_X(x)\\
&=\sum_{q=-M+1}^\infty\int_{\{q\leq \mu<q+1\}}\,P_p^{(m)}(x,x)\,dv_X(x)
\leq C_k\,p^{n-\frac{k}{2}}\,.
\end{split}
\end{equation}
Taking $k>2n$, we can choose $p_0\in\N$ such that for all $p\geq p_0$
and all $m\in\Z$ satisfying $m<-M\,p$,
we have $C_k\,p^{n-\frac{k}{2}}<1$, so that $\dim H^0_{(2)}(X,L^p)_m=0$.
This implies the identities \cref{QR=0fla} and \cref{Hpfin}
by definition
\cref{Hpdef} of the subspace $\HH_p\subset H^0_{(2)}(X,L^p)$.


Finally, to show that the space \cref{Hpdef} is
finite-dimensional, it suffices to show that
for any $m\leq 0$, the weight space
$H^0_{(2)}(X,L^p)_m$ is finite-dimensional.
Fix then $m\in\N$ such that $m\leq 0$.
For any $k\in\N$, 
\cref{oscvener,volgrowth} imply as in \cref{Npcomput0}
the existence of a constant $\delta>0$ such that for any $k\in\N$,
there exists a constant
$C_k>0$ such that for any $p\in\N$
and any $q\in\N$ satisfying $q\geq 1$, we have
\begin{equation}\label{Npcomput3pre}
\int_{\{q\leq \mu<q+1\}}\,P_p^{(m)}(x,x)
\,dv_X(x)\leq C_k\,p^{-k}
(q+1)^{N} q^{-\delta k}\,.
\end{equation}
Taking $k>(N+2)/\delta$, we obtain that the
sum of \cref{Npcomput3pre}
for all $q\in\N$ satisfying $q\geq 1$ converges,
so that $\dim H^0_{(2)}(X,L^p)_m<+\infty$
by the first line of \cref{dimHm} and
using the fact that $\{\mu\leq 1\}\subset X$ is
compact.
\end{proof}

Let us now fix $p\in\N$ large enough so that
the conclusions of \cref{QR=0} hold,
set $N_p:=\dim\HH_p\in\N$,
and consider an orthonormal basis
$\{s_j\in\HH_p\}_{j=0}^{N_p}$ with respect
to the $L^2$-Hermitian product \cref{L2}.
We define the \emph{Slater determinant}
$\det s_p\in\cinf(X^{N_p},(L^p)^{\boxtimes N_p})$
as the section of $(L^p)^{\boxtimes N_p}$
over the $N_p$-fold product $X^{N_p}$ given
for any $(x_1,x_2,\cdots x_{N_p})\in X^{N_p}$
by
\begin{equation}\label{Slaterdef}
(\det s_p)(x_1,x_2,\cdots x_{N_p})
:=\det(s_j(x_i))_{i,\,j=1}^{N_p}\,,
\end{equation}
which does not depend on the choice of orthonormal
basis of $\HH_p$.
Let us endow $(L^p)^{\boxtimes N_p}$ with the
Hermitian metric induced by $h^L$, and write
$|\cdot|_p$ for the induced pointwise norm.
The following fundamental Lemma
from the basic theory of
\emph{determinantal point processes}
can be found for instance in \cite[\S.\,4.2.3]{AGZ10},
following the adaptation to prequantized Kähler manifolds
due to Berman in \cite[\S\,5.1]{Ber18}.

\begin{lem}\label{dmupdef}
For any $p\in\N$ and
any $(x_1,x_2,\cdots x_{N_p})\in X^{N_p}$,
the Slater determinant \cref{Slaterdef}
satisfies the following formula,
\begin{equation}\label{Slaterfla}
|(\det s_p)(x_1,x_2,\cdots x_{N_p})|_p^2
=\det(P_p^{(-)}(x_i,x_j))_{i,\,j=1}^{N_p}\,,
\end{equation}
and the measure $d\nu_{N_p}$ over $X_{N_p}$ given by
\begin{equation}\label{dmupfla}
d\nu_{N_p}:=\frac{1}{N_p!}\,
\left|\det s_p\right|_p^2~dv_X^{N_p}\,,
\end{equation}
defines a probability measure over $X_{N_p}$, called
the \emph{determinantal point process} associated with
the Slater determinant \cref{Slaterdef}.
\end{lem}
\begin{proof}
Let us first recall that, by
standard properties of the Schwartz
kernel of the orthogonal projection $P_p^{(-)}:L^2(X,L^p)\to\HH_p$
on the subspace $\HH_p\subset L^2(X,L^p)$ of finite dimension
$N_p\in\N$,
for any orthonormal basis $\{s_j\in\HH_p\}_{j=1}^{N_p}$ and any $x,\,y\in X$, we have
\begin{equation}
P_p^{(-)}(x,y)=\sum_{j=1}^{N_p}\,s_j(x)\otimes s_j(y)^*\in L^p_x\otimes(L^p_y)^*\,.
\end{equation}
One then readily compute for any
$(x_1,x_2,\cdots x_{N_p})\in X^{N_p}$ that
\begin{equation}
\begin{split}
|(\det s_p)(x_1,x_2,\cdots x_{N_p})|_p^2&=
\det(s_j(x_i))_{i,\,j=1}^{N_p}\,
\overline{\det(s_j(x_i))_{i,\,j=1}^{N_p}}\\
&=\det(P_p^{(-)}(x_i,x_j))_{i,\,j=1}^{N_p}\,,
\end{split}
\end{equation}
which establishes \cref{Slaterfla}.

On the other hand,
the fact that \cref{dmupfla} defines a probability measure
is a straightforward consequence of the following
\emph{Andreief formula} from \cite[(5.8)]{Ber18},
valid for any basis $\{s_j\in\HH_p\}_{j=1}^{N_p}$ and
which can be adapted for instance
from the proof of \cite[Lem.\,3.2.3]{AGZ10},
\begin{equation}
\int_{X^{N_k}}\,
\left|\det(s_j(x_i))_{i,\,j=1}^{N_p}\right|_p
\,dv_X(x_1)\cdots dv_X(x_{N_k})=
N_p!\,\det\left(\<s_i,s_j\>_p\right)_{i,\,j=1}^{N_p}\,.
\end{equation}
Since the basis
$\{s_j\in\HH_p\}_{j=1}^{N_p}$
is orthonormal, this establishes the result.
\end{proof}

For any bounded measurable function $f\in L^\infty(X,\R)$
and any $p\in\N$,
let us consider the associated
\emph{linear statistics}, which is the random variable
over $(X^{N_p},d\nu_{N_p})$ defined by the function
\begin{equation}\label{linstat}
\begin{split}
\NN_p[f]:X^{N_p}&\longrightarrow\R\\
(x_j)_{j=1}^{N_p}&\longmapsto\sum_{j=1}^{N_p}\,f(x_j)\,.
\end{split}
\end{equation}
Recall that
for all $x\in X$, we have $P_p(x,x)>0$
through the canonical identification $L^p_x\otimes(L^p_x)^*\simeq\C$.
The following explicit formulas
for the expectation $\IE[\NN_p[f]]$ and the variance $\IV[\NN_p[f]]$
of the random variable \cref{linstat} are straightforward concequences
of the explicit formula for correlations functions of
determinantal point processes as defined for instance in
\cite[Def.\,4.2.1]{HKPV09}, which can be obtained following
the proof of \cite[Lem.\,4.5.1]{HKPV09}.

\begin{prop}\label{ExpVarprop}
{\cite[Lem.\,6.2]{Ber18}}
For any $f\in L^\infty(X,\R)$ and $p\in\N$,
the expectation and the variance
of the random variable \cref{linstat} with respect to the determinantal point
process \cref{dmupfla} are given by
\begin{multline}
\IE[\NN_p[f]]=\int_X\,P_p^{(-)}(x,x)\,f(x)\,dv_X(x)~\text{ and }~\\
\IV[\NN_p[f]]=\frac{1}{2}\int_X\int_X\,\left|P_p^{(-)}(x,y)\right|^2_p
\,(f(x)-f(y))^2\,dv_X(x)\,dv_X(y)\,.
\end{multline}
\end{prop}

\begin{ex}\label{Ginibreex}
Let us consider the model case $X=\C$ equipped with the trivial Hermitian line
bundle $(L,h^L)=(\C,|\cdot|)$, endowed with the Hermitian connection $\nabla^L$
defined for all $z\in\C$ by
\begin{equation}
\nabla^L:=d+\frac{\pi}{2}(zd\bar{z}-\bar{z}dz)\,.
\end{equation}
Its curvature satisfies the prequantization formula \cref{preq}
for the standard symplectic form $\om=\frac{\sqrt{-1}}{2}dz\wedge d\bar{z}$ on $X=\C$, and
for any $p\in\N$, the Cauchy-Riemann operator associated
with the induced holomorphic structure on $L^p=\C$ is given by
$\nabla^{L^p}_{\frac{\partial}{\partial\bar{z}}}=\frac{\partial}{\partial\bar{z}}+p\,\frac{\pi}{2} z$, so that the global holomorphic sections of $L^p$
are exactly the smooth sections $s\in\cinf(X,L^p)$ of the form
$s(z)=f(z)\,e^{-p\frac{\pi}{2}|z|^2}$ for all $z\in\C$, where $f\in H^0(\C,\C)$
is holomorphic. Let now $S^1\subset\C$ act both on $X=\C$ and on $L=\C$
by complex multiplication, so that for any holomorphic section
$s\in H^0(X,L^p)$ as above, for all $t\in S^1\simeq\R/\Z$ and $z\in\C$, we have
\begin{equation}
\varphi_t^*s(z)=
e^{-2\pi\sqrt{-1} pt}f(e^{2\pi\sqrt{-1}t}z)e^{-p\frac{\pi}{2}|z|^2}\,.
\end{equation} 
We thus see that for all $m\in\Z$ satisfying $m\geq -p$,
the associated weight space \cref{wghtspace} is the subspace
 $H^0_{(2)}(X,L^p)_m\subset H^0_{(2)}(X,L^p)$ generated by the section
$s_m\in H^0_{(2)}(X,L^p)$ defined for all $z\in\C$ by
\begin{equation}
s_m(z)=z^{m-p}\,e^{-p\frac{\pi}{2}|z|^2}\,,
\end{equation}
while $H^0_{(2)}(X,L^p)_m=\{0\}$ for all $m\in\Z$ satisfying $m<-p$.
This shows that the finite dimensional
space \cref{Hpdef} used to define the
determinantal point process \cref{DPPholdef} considered in \cref{mainth}
coincides with the subspace $\HH_p\subset \cinf(X,L^p)$ generated
by the orthonormal families of functions \cref{Gindef}
used to define the \emph{Ginibre ensemble}, after the
change of variable $z\mapsto z/\sqrt{\pi}$ and setting $N:=p+1$.

On the other hand, the associated
Kostant moment map $\mu\in\cinf(X,\R)$ as in
\cref{Kostantmudef}
is given for any $z\in\C$ by
$\mu(z)=\pi\,|z|^2-1$. Thus after the change
of variable $z\mapsto z/\sqrt{\pi}$, the law of large numbers
\cref{LLNmainth} in
\cref{mainth} recovers the classical result
of Ginibre in \cite[\S\,1]{Gin65}, stating that this
determinantal point process admits the measure
$d\nu:=\mathds{1}_{\mathbb{D}}\,dv_\C$
as an equilibrium measure, while the
asymptotics \cref{Varmainth}
for its variance recovers the asymptotics \cref{Varintro}
of Rider and Virag in \cite{RV07}
describing the behavior of its fluctuations, as well as the
corresponding central limit theorem.
\end{ex}

\subsection{Law of large Numbers}
\label{LLNsec}

Consider the setting of \cref{Gensec},
and for any $f\in L^\infty(X,\R)$,
let us consider the linear statistics \cref{linstat}
as a random variable with respect to the determinantal
point process introduced in \cref{dmupdef}.
The first part of \cref{mainth} is a consequence of the
following \emph{law of large numbers}, which
is the first main result of this section.

\begin{theorem}\label{LLN}
For any $f\in L^\infty(X,\R)$,
the expectation of the linear
statistics \cref{linstat} is finite for all
$p\in\N$, and
satisfies the following asymptotics as $p\to\infty$,
\begin{equation}\label{Expfla}
\IE\left[\NN_p[f]\right]=p^n\,\int_{\{\mu<0\}}\,f(x)\,dv_X(x)
+o(p^{n})\,.
\end{equation}
Furthermore, we have the following convergence
in probability as $p\to\infty$,
\begin{equation}\label{LLNfla1}
\frac{1}{N_p}\NN_p[f]\xrightarrow{~p\to\infty~}~
\frac{1}{\Vol\left(\{\mu<0\}\right)}\int_{\{\mu<0\}}\,f\,dv_X\,.
\end{equation}
If in addition $f\in L^\infty(X,\R)$
is continuous with compact support,
there exists $C>0$ such that
for any $\epsilon>0$ and $p\in\N$, we have
\begin{equation}\label{LLNfla2}
\IP\left(\left\{(x_j)_{j=1}^{N_p}\in X^{N_p}~\Bigg|~\left|\frac{1}{N_p}\NN_p[f]-\frac{1}{\Vol\left(\{\mu<0\}\right)}\int_{\{\mu<0\}}\,f\,dv_X\right|>\epsilon\right\}\right)\leq\frac{C}{\epsilon p^n}\,.
\end{equation}
\end{theorem}
\begin{proof}

Recall that $\mu\in\cinf(X,\R)$ has polynomial growth in the
sense of \cref{volgrowth},
let $\epsilon_0>0$ be as in definition \cref{psidef}
of the smooth family of charts used in
\cref{coordsec,partBergsec}, and pick
$\epsilon\in\,]0,\epsilon_0[$. Recall also that
for any $p\in\N$ and all $x\in X$, we have $P_p(x,x)>0$
through the canonical identification
$L^p_x\otimes(L^p_x)^*\simeq\C$.
By \cref{ExpVarprop}, for any $p\in\N$
and any positive function $f\in L^\infty(X,\R)$, we have
the following identity in $[0,+\infty]$,
\begin{multline}\label{Npcomput}
\IE[\NN_p[f]]=\int_X\,P_p^{(-)}(x,x)\,f(x)\,dv_X(x)\\
=\int_{\{\mu<-\ept\}}\,P_p^{(-)}(x,x)\,f(x)\,dv_X(x)+
\int_{\{|\mu|\leq\ept\}}\,P_p^{(-)}(x,x)\,f(x)\,dv_X(x)\\
+
\int_{\{\mu>\ept\}}\,P_p^{(-)}(x,x)\,f(x)\,dv_X(x)\,.
\end{multline}
We will show that all terms of \cref{Npcomput}
are finite, and that
the last one is negligible as $p\to\infty$.
Let us first deal with the last term.
By \cref{faroffdiagprop} and the properness
of $\mu\in\cinf(X,\R)$, for any $k\in\N$,
we get a constant $C_k>0$ such that for any $p\in\N$, we have
\begin{equation}\label{Npcomput2}
\int_{\{\ept<\mu<2\}}\,P_p^{(-)}(x,x)\,f(x)\,dv_X(x)
\leq C_k\,p^{-k}\,\sup\limits_{x\in X}f(x)\,,
\end{equation}
while on the other hand,
using \cref{QR=0},
we can apply \cref{oscvener} as in formula \cref{Npcomput3pre}
to show that there exist $M,\,N\in\N$ and $\delta>0$
such that for any $q,\,p\in\N$ satisfying $q\geq 1$, we have
\begin{equation}\label{Npcomput3}
\begin{split}
\int_{\{q\leq \mu<q+1\}}\,P_p^{(-)}(x,x)\,f(x)\,dv_X(x)
&=\sum_{m= -p\,M}^{m=0}\,\int_{\{q\leq \mu<q+1\}}\,P_p^{(m)}(x,x)
\,f(x)\,dv_X(x)\\
&\leq M\,C_k\,p^{-k+1}
(q+1)^{N} q^{-\delta k}\,\sup\limits_{x\in X}f(x),
\end{split}
\end{equation}
Taking $k\in\N$ large enough and summing  \cref{Npcomput2}
with \cref{Npcomput3} over all $q\in\N$ satisfying $q>2$,
for any $k\in\N$, we get a constant $C_k>0$ such that for
any $p\in\N$, we have the following estimate,
\begin{equation}\label{intfaroff}
\int_{\{\mu>\ept\}}\,P_p^{(-)}(x,x)\,f(x)\,dv_X(x)\leq C_k\,
p^{-k}\,\sup\limits_{x\in X}f(x)\,.
\end{equation}
This implies that the third term of
formula \eqref{Npcomput} for the expectation
$\IE[\NN_p[f]]\in[0,+\infty]$ is negligible
as $p\to\infty$, and since $\{\mu<\ept\}\subset X$
is compact by \cref{volgrowth}, it
implies in particular that the expectation
$\IE[\NN_p[f]]>0$ is finite. 

To deal with the second term of \cref{Npcomput},
recall that $\{|\mu|\leq\ept\}\subset X$
is compact and
formed of regular values of $\mu\in\cinf(X,\R)$
by definition of $\epsilon>0$ from
\cref{phidef,psidef}. Then
\cref{partBergasyth}
implies
the existence of a constant
$C>0$ such that for all $p\in\N$ and all $x\in X$,
we have 
\begin{equation}\label{intedge}
\int_{\{|\mu|\leq\ept\}}\,P_p^{(-)}(x,x)\,f(x)\,dv_X(x)\leq\,C\,p^{n}\ptheta\,\sup\limits_{x\in X}f(x)\,,
\end{equation}
Hence from \cref{asydiag}, \cref{faroffdiagprop}
and the estimate \cref{intfaroff}
for the third term of
\cref{Npcomput},
we get the following asymptotic estimate as $p\to\infty$,
\begin{equation}\label{Npcomputfinal}
\begin{split}
p^{-n}\,\IE[\NN_p[f]]
&=p^{-n}\int_{\{\mu<-\ept\}}\,P_p(x,x)\,f(x)\,dv_X(x)+O(\ptheta)\\
&=\int_{\{\mu<-\ept\}}\,f(x)\,dv_X(x)+O(\ptheta)\,.
\end{split}
\end{equation}
Using once more that $\{|\mu|\leq\ept\}\subset X$
is compact and formed of regular values of
$\mu\in\cinf(X,\R)$, this
implies the asymptotics \cref{Expfla} for the
expectation of the linear statistics \cref{linstat}
of a bounded positive function $f\in L^\infty(X,\R)$,
hence of any $f\in L^\infty(X,\R)$ by linearity.
%
%
%

Let us now establish the law of large numbers
\eqref{LLNfla1}.
Using the fact that for all $p\in\N$, we have
$N_p=\IE[\NN_p[f]]$ by definition of the linear statistics
\cref{linstat},
formula \cref{Expfla} gives
\begin{equation}
\IE\left[\frac{1}{N_p}\NN_p[f]\right]
\xrightarrow{~p\to\infty~}~\frac{1}{\Vol\left(\{\mu<0\}\right)}\int_{\{\mu<0\}}\,f\,dv_X\,.
\end{equation}
On the other hand, the identity for the 
variance of $\NN_p[f]$ given in \cref{ExpVarprop}
together with
elementary properties of the smooth Schwartz
kernel of the orthogonal projection $P_p^{(-)}:L^2(X,L^p)\to\HH_p$ gives
the following estimates for all $p\in\N$,
\begin{equation}
\begin{split}
\IV[\NN_p[f]]&=\frac{1}{2}\int_X\int_X\,\left|P_p^{(-)}(x,y)\right|^2_p
\,(f(x)-f(y))^2\,dv_X(x)\,dv_X(y)\\
&\leq 2\,\sup_{x\in X}\,|f(x)|^2\,\int_X\int_X\,P_p^{(-)}(x,y).P_p^{(-)}(y,x)
\,dv_X(y)\,dv_X(x)\\
&\leq 2\,\sup_{x\in X}\,|f(x)|^2\,\int_X\,P_p^{(-)}(x,x)
\,dv_X(x)\leq 2\,\sup_{x\in X}\,|f(x)|^2\,N_p\,.
\end{split}
\end{equation}
Together with formula \cref{Expfla} applied to
$N_p=\IE[\NN_p[1]]$ and elementary properties
of the variance,
this implies that the existence of a constant
$C>0$ such that for all $p\in\N$, we have
\begin{equation}
\begin{split}
\IV\left[\frac{1}{N_p}\NN_p[f]\right]&=\frac{1}{N_p^2}\IV\left[\NN_p[f]\right]\\
&\leq C\,p^{-n}\,.
\end{split}
\end{equation}
The convergence in probability \cref{LLNfla1} then follows from the classical Chebyshev inequality, as in the standard proof of the weak law of large numbers
for random variables with finite variance.

To get the more precise result \cref{LLNfla2} when
$f\in\CC^0_c(X,\R)$ is continuous
with compact support,
let us use the identity for the 
variance of $\NN_p[f]$ given in \cref{ExpVarprop} to
split it into two parts in the following way,
for all $p\in\N$,
\begin{equation}\label{Varcomput}
\begin{split}
&\IV\left[\NN_p[f]\right]
=\frac{1}{2}\int_X\int_X\,|P_p^{(-)}(x,y)|^2_p\,(f(x)-f(y))^2\,dv_X(x)\,dv_X(y)\\
&=\frac{1}{2}\int_{\{\mu<\ept\}}\int_X\,|P_p^{(-)}(x,y)|^2_p\,(f(x)-f(y))^2\,dv_X(x)\,dv_X(y)\\
&+
\frac{1}{2}\int_{\{\mu>\ept\}}\int_X\,|P_p^{(-)}(x,y)|_p^2\,(f(x)-f(y))^2\,dv_X(x)\,dv_X(y)\,.
\end{split}
\end{equation}
To deal with the second term, we can apply the estimate
\cref{intfaroff} and use elementary properties of the Schwartz
kernel of the orthogonal projection $P_p^{(-)}:L^2(X,L^p)\to\HH_p$
to get for any $k\in\N$ a constant $C_k>0$ such that
for all $p\in\N$, we have
\begin{equation}\label{Varcomput2}
\begin{split}
&\frac{1}{2}\int_{\{\mu>\ept\}}\int_X\,|P_p^{(-)}(x,y)|_p^2\,(f(x)-f(y))^2\,dv_X(x)\,dv_X(y)\\
&\leq\, 2\,\sup_{x\in X}\,|f(x)|^2\,\int_{\{\mu>\ept\}}\left(\int_X\,P_p^{(-)}(x,y).P_p^{(-)}(y,x)
\,dv_X(y)\right)\,dv_X(x)\\
&\leq\, 2\,\sup_{x\in X}\,|f(x)|^2\,\int_{\{\mu>\ept\}}\,P_p^{(-)}(x,x)\,dv_X(x)
\leq 2\,C_k\,p^{-k}\,\sup_{x\in X}\,|f(x)|^2\,.
\end{split}
\end{equation}
%
On the other hand, for any $x\in X$,
we can apply \cref{theta} and use the uniform continuity of
$f\in\CC^0_c(X,\R)$ over its compact support
to get, for any $\delta>0$ and $k\in\N$, the existence of
constants $C,\,C_k>0$ and $p_0\in\N$
such that for any $p\geq p_0$ and $x\in X$, we have
\begin{multline}\label{BulkLLN}
\frac{1}{2}\int_{X}\,|P_p(x,y)|^2_p\,(f(x)-f(y))^2\,dv_X(y)=\frac{1}{2}\,
\int_{{B^X(x,\ept)}}\,|P_p(x,y)|^2_p\,(f(x)-f(y))^2\,dv_X(y)\\
+\frac{1}{2}\,\int_{X\backslash {B^X(x,\ept)}}\,|P_p(x,y)|^2_p\,(f(x)-f(y))^2\,dv_X(y)\\
\leq\frac{1}{2}
\Bigg(\sup_{y\in B^X(x,\ept)}\,|f(x)-f(y)|^2\Bigg)
\int_{X}\,|P_p(x,y)|^2_p\,
\,dv_X(y)\\
+C_k\,p^{-k}\int_X\,|P_p(x,y)|_p\,|f(x)-f(y)|^2\,dv_X(y)\\
\leq\delta\,P_p(x,x)+C\,C_k\,p^{n-k}\,2\,\sup_{x\in X}\,|f(x)|^2\,\Vol(\Supp\,f)\\
\leq Cp^n\left(\delta+C_k\,p^{-k}\,2\,\sup_{x\in X}\,|f(x)|^2\,\Vol(\Supp\,f)\,\right)\,.
\end{multline}
Recalling that $\mu\in\cinf(X,\R)$ is proper and bounded from below,
using formula \cref{Expfla} applied to
$N_p=\IE[\NN_p[1]]$ and elementary properties of
the variance,
we can then plug the estimates \cref{BulkLLN,Varcomput2}
inside \cref{Varcomput} to get a constant $C>0$
such that for any
$\delta>0$, there exists $p_0\in\N$ such that for all $p\geq p_0$,
we have
\begin{equation}\label{Varcomput3}
\IV\left[\frac{1}{N_p}\NN_p[f]\right]
\leq C\,p^{-n}\Vol(\{\mu<\epsilon\})
\,\delta\,.
\end{equation}
The inequality \cref{LLNfla2} in then a direct consequence
of the classical Chebyshev inequality.

\end{proof}

\subsection{Central Limit Theorem}
\label{CLTsec}


Consider the setting of \cref{Gensec}, and for any
$f\in L^\infty(X,\R)$,
let us consider the linear statistics \cref{linstat}
as a random variable with respect to the determinantal
point process introduced in \cref{DPPdef}.
Together with \cref{LLN}, the following
result concludes the proof of \cref{mainth}.

\begin{theorem}\label{Varth}
For any $f\in\cinf_c(X,\R)$, the variance of the associated linear statistics
\cref{linstat} satisfies the following asymptotics as $p\to\infty$,
\begin{equation}\label{Varfla}
\lim\limits_{p\to\infty}\,p^{-n+1}
\IV\left[\NN_p[f]\right]=
\frac{1}{4\pi}\int_{\{\mu<0\}}\,|df|^2\,dv_X(x)+
\frac{1}{2}\int_{X_0}\,\sum_{k\in\Z}\,|k|\,
|\hat{f}_k(x)|^2\,dv_{X_0}(x)\,,
\end{equation}
and the random variable $N_p^{\alpha}(\NN_p[f]-\IE[\NN_p[f]])$
with $\alpha=\frac{1}{2n}-\frac{1}{2}$
converges in distribution to a centered
normal random variable with variance \cref{Varfla}
as $p\to\infty$.
\end{theorem}
\begin{proof}
Recall that $\mu\in\cinf(X,\R)$ has polynomial growth in the
sense of \cref{volgrowth},
and fix $f\in\cinf_c(X,\R)$.
Let $\epsilon_0>0$ be as in definition \cref{psidef}
of the smooth family of charts used in
\cref{coordsec,partBergsec}, and pick
$\epsilon\in\,]0,\epsilon_0[$.
Recalling the decomposition \cref{Varcomput} of the formula
for the variance of $\NN_p[f]$ given by \cref{ExpVarprop},
the estimate \cref{Varcomput2}
implies the following asymptotics as $p\to\infty$,
\begin{multline}\label{Vardec}
p^{-n+1}\,\IV\left[\NN_p[f]\right]=
\frac{1}{2}\,p^{-n+1}\int_{\{\mu<-\ept \}}\int_X\,|P_p^{(-)}(x,y)|^2_p\,(f(x)-f(y))^2\,dv_X(x)\,dv_X(y)\\
+
\frac{1}{2}\,p^{-n+1}\int_{\{|\mu|\leq\ept\}}\int_X\,|P_p^{(-)}(x,y)|_p^2\,(f(x)-f(y))^2\,dv_X(x)\,dv_X(y)+O(p^{-\infty})\,.
\end{multline}
We claim that the first and second terms of the decomposition \cref{Vardec}
respectively converge to the first and second terms of the limit \cref{Varfla}
as $p\to\infty$.

To deal with the first term of decomposition \cref{Vardec},
first note from \cref{faroffdiagprop} that, since
$f\in\cinf_c(X,\R)$ has compact support, we have the
following estimate as $p\to\infty$,
\begin{multline}\label{CLTcomput1}
\int_{\{\mu<-\ept \}}\int_X\,|P_p^{(-)}(x,y)|^2_p\,(f(x)-f(y))^2\,dv_X(x)\,dv_X(y)\\
=\int_{\{\mu<-\ept \}}\int_X\,|P_p(x,y)|^2_p\,(f(x)-f(y))^2\,dv_X(x)\,dv_X(y)+O(p^{-\infty})\,.
\end{multline}
Consider now a bounded family of charts
$\psi:B^{TX}(0,\epsilon_0)\to X$ in the sense of
\cref{chart}, and note that
the Riemannian volume
form satisfies $\psi_x^*dv_X=(1+O(|Z|))\,dZ$
uniformly in $x\in\Supp\,f$
for all $Z\in T_xX$ satisfying $|Z|<\epsilon_0$,
where $dZ$ is the Lebesgue measure of $(T_xX,|\cdot|)$.
Let now $x\in X$, pick an orthonormal basis
$(e_j\in T_xX)_{j=1}^n$, and for any $Z\in T_xX$, write
$Z=\sum_{j=1}^n\,Z_j\,e_j$ with $Z_j\in\R$ for all $1\leq j\leq n$.
We can then apply \cref{theta} as in \cref{thetagal}
together with \cref{asy} to get $\epsilon_0>0$ such that
for all $\epsilon\in\,]0,\epsilon_0[$,
the following estimates hold as $p\to\infty$,
\begin{equation}\label{CLTcomput2}
\begin{split}
&\frac{1}{2}\,p^{-n+1}\int_{X}\,|P_p(x,y)|^2_p\,(f(x)-f(y))^2\,dv_X(y)\\
&=
\frac{1}{2}\,p^{-n+1}\int_{B^{T_xX}(0,\ept)}\,|P_{p,x}(0,Z)|^2_p
\,(f_x(0)-f_x(Z))^2\,dv_X(y)+O(p^{-\infty})\\
&=\frac{1}{2}\,p^{n+1}\int_{B^{T_xX}(0,\eptt)}\,
|P_{p,x}(0,Z/\sqrt{p})|^2_p
\,\left(\sum_{j=1}^{2n} \Big(\frac{Z_j}{\sqrt{p}}\,\Dkk{}{f_x}{Z_j}(0)\Big)^2+O(|Z/\sqrt{p}|^3)\right)\,dZ+
O(p^{-\infty})\\
&=\frac{1}{2}\,\sum_{j=1}^{2n}\Big(\Dkk{}{f_x}{Z_j}(0)\Big)^2\int_{\R^{2n}}\,\exp(-\pi\,|Z|^2)\,Z_j^2\,dZ+o(1)=\frac{1}{4\pi}\,|df_x|^2+o(1)\,,
\end{split}
\end{equation}
and this estimate is uniform in
$x\in X$ since $f\in\cinf_c(X,\R)$ has compact support.
Using further that for all $\epsilon\in\,]0,\epsilon_0[$
and $p\in\N$, the subset $\{|\mu|\leq\ept\}\subset X$
is compact and formed of regular values of
$\mu\in\cinf(X,\R)$, we deduce from \cref{CLTcomput1}
and \cref{CLTcomput2}
the following estimates for
the first term in the decomposition \cref{Vardec}
as $p\to\infty$,
\begin{multline}\label{termtrivial}
\frac{1}{2}\,p^{-n+1}\int_{\{\mu<-\ept \}}\int_X\,|P_p^{(-)}(x,y)|^2_p\,(f(x)-f(y))^2\,dv_X(x)\,dv_X(y)\\
=\frac{1}{4\pi}\int_{\{\mu<0\}}|df|^2\,dv_X(x)+o(1)\,.
\end{multline}
This identifies the first term of the right-hand side of
the asymptotics
\cref{Vardec} with the first term of the right-hand side of
the asymptotics \cref{Varfla}.

Let us now focus on the second term in the decomposition \cref{Vardec}
for the variance. 
To that end, let $x_0\in\mu^{-1}(0)$, and
recall the local section $V_0\subset\mu^{-1}(0)$
of the $S^1$-principal bundle $\pi:\mu^{-1}(0)\to X_0:=\mu^{-1}(0)/S^1$ defined in \cref{BNdef}.
Then the image $V\subset X$ of the embedding
$\phi:I_0\times V_0\to X$ defined by \cref{phidef}
is a local section of the quotient map
\begin{equation}\label{princbdle}
\pi:U\to B:=U/S^1\,,
\end{equation}
where the open set $U\subset X$ is the image of the
open embedding
$\Phi:S^1\times I_0\times V_0\to X$ defined by \cref{eqemb}.
From the definition \cref{THX} of the horizontal
tangent bundle, we see that for any $x\in U$,
the differential
$d\pi_x:T_xX\to T_{\pi(x)}B$ induces an isomorphism
$T_x^HX\simeq T_{\pi(x)}B$ by restriction.
Since the $S^1$-action preserves the Riemannian
metric $g^{TX}$, there is a unique Riemannian metric $g^{B}$ on $B$
such that $\pi^*g^{B}=g^{TX}|_{T^HX}$, and
recalling the fundamental vector field \cref{xidef},
the volume form $dv_{B}$ of $(B,g^{B})$ satisfies
\begin{equation}\label{dvX=dtdvB}
dv_X=|\xi|\,dt\,\pi^*dv_{B}\,,
\end{equation}
where $dt$ is the Lebesgue measure of $S^1\simeq\R/\Z$.

Let now $\varrho\in\cinf(X_0,\R)$ have compact support in
$\pi(V_0)\subset X_0$, and let
$\widetilde{\varrho}\in\cinf_c(X,\R)$ be
the $S^1$-invariant function with compact
support in $U\subset X$
defined
for
any $t\in S^1$, $u\in I_0$ and $x\in V_0$ by
\begin{equation}\label{rhotilde}
\widetilde{\varrho}\,(\Phi(t,u,x)):=\varrho\,(\pi(x))\,.
\end{equation}
For any $\epsilon\in\,]0,\epsilon_0[$ , write
\begin{equation}\label{Veps}
V(\epsilon):=V\,\cap\,\{|\mu|\leq\epsilon\}\subset U\,.
\end{equation}
Then for any $\epsilon\in\,]0,\epsilon_0[$ and $p\in\N$,
since $S^1$ acts unitarily on $(L,h^L)$ via $\varphi_{t,p}:L^p\to L^p$ for all
$t\in S^1\simeq\R/\Z$,
we get that
\begin{multline}\label{Edgecomput1}
\int_{\{|\mu|\leq\ept\}}
\int_X\,|P_p^{(-)}(x,y)|^2_p\,(f(x)-f(y))^2
\,\widetilde{\varrho}(x)\,dv_X(x)\,dv_X(y)\\
=\int_{V(\ept)}\int_V\,
\int_{S^1}\int_{S^1}|\varphi_{t,p}^{-1} P_p^{(-)}(\varphi_t(x),\varphi_u(y))\varphi_{u,p}|^2_p\\
(f(\varphi_t(x))-f(\varphi_u(y)))^2\,
\widetilde\varrho(x)
\,|\xi_x|\,|\xi_y|\,dt\,du\,\pi^*dv_{B}(x)\,\pi^*dv_{B}(y)
\,.
\end{multline}
Now for
all $x,\,y\in X$ and $t,\,u\in\R$, through the canonical isomorphism 
$L^p_x\otimes(L^p_x)^*\simeq\C$,
\cref{Pjprop,QR=0} imply the existence of $M\in\N$ such that
\begin{multline}
|\varphi_{t,p}^{-1} P_p^{(-)}(\varphi_t(x),\varphi_u(y))\varphi_{u,p}|^2_p\\
=
\left(\varphi_{t,p}^{-1} P_p^{(-)}(\varphi_t(x),\varphi_u(y))\varphi_{u,p}\right).\left(\varphi_{t,p}^{-1}\,P_p^{(-)}(\varphi_t(y),\varphi_u(x))\varphi_{u,p}\right)\\
=\sum_{m,\,r=-Mp}^{m,\,r=0}
e^{2\pi\sqrt{-1}t(m-r)}e^{2\pi\sqrt{-1}u(r-m)}\,P^{(m)}(x,y).P^{(r)}(y,x)\,.
\end{multline}
By standard Fourier theory from notation
\cref{Fouriercoeff} for the Fourier coefficients,
this implies that for any
functions $g,\,h\in\cinf(X,\C)$ and any $x,\,y\in U$, we have
\begin{multline}\label{Fourierdev}
\int_{S^1}\int_{S^1}|\varphi_{t,p}^{-1} P_p^{(-)}(\varphi_t(x),\varphi_s(y))\varphi_{u,p}|^2_p\,g(\varphi_t(x))\,h(\varphi_u(y))\,dt\,du\\
=\sum_{m,\,r=-Mp}^{m,\,r=0}\,P^{(m)}(x,y).P^{(r)}(y,x)\,\hat{g}_{r-m}(x)\,
\hat{h}_{m-r}(y)\,.
\end{multline}
Furthermore, since $f\in\cinf_c(X,\R)$
is smooth and compactly supported for any $r\in\Z$,
we have that
$\widehat{f^2_r}(x)=
\sum_{k\in\Z}\hat{f}_k(x)\hat{f}_{r-k}(x)$ with uniform
convergence in $x\in U$, and that
$\widehat{f}_{r}(x)=\overline{\widehat{f}_{-r}(x)}$ since $f$ assumes real values.
Then applying formula \cref{Fourierdev}
to each term of $(f(x)-f(y))^2=f(x)^2-2f(x)f(y)
+f(y)^2$ for any $x,\,y\in U$ and using
the fact from \cref{QR=0} that $P^{(m)}_p\equiv 0$ for all $m<-Mp$, one gets
\begin{multline}\label{Fourierdev2}
\int_{S^1}\int_{S^1}|\varphi_{t,p}^{-1} P_p^{(-)}(\varphi_t(x),\varphi_u(y))\varphi_{u,p}|^2_p\,\left(f(\varphi_t(x))-f(\varphi_u(y))\right)^2\,dt\,du\\
=\sum_{m=-Mp}^{m=0}\,|P^{(m)}(x,y)|_p^2\,\sum_{k\in\Z}\,\left(
|\hat{f}_k(x)|^2+|\hat{f}_k(y)|^2\right)\\
-2\sum_{m,\,r=-Mp}^{m,\,r=0}\,P^{(m)}(x,y).P^{(r)}(y,x)\,
\hat{f}_{r-m}(x)\,
\overline{\hat{f}_{m-r}(y)}\\
=\sum_{m=-Mp}^{m=0}\,|P^{(m)}(x,y)|_p^2\,\sum_{k\in\Z}\left|\hat{f}_k(x)
-\hat{f}_{k}(y)\right|^2\\
-2\sum_{k\leq 0}\sum_{m=-Mp}^{m=0}\left(\,
 P^{(m)}(x,y).P^{(m+k)}(y,x)-|P^{(m)}(x,y)|_p^2\right)\,\hat{f}_{k}(x)\,
\overline{\hat{f}_{k}(y)}\\
-2\sum_{k<0}\,\sum_{m=-Mp}^{m=0}\left(\,
 P^{(m+k)}(x,y).P^{(m)}(y,x)-|P^{(m)}(x,y)|_p^2\right)\,\hat{f}_{-k}(x)\,
\overline{\hat{f}_{-k}(y)}\,,
\end{multline}
where all sums in $k\in\Z$ converge uniformly in $x,\,y\in U$.
Plugging \cref{Fourierdev2} into the expression \cref{Edgecomput1}, we then get the following identities,
\begin{multline}\label{termscomput}
\frac{1}{2}\,p^{-n+1}\int_{\{|\mu|<\ept\}}
\int_X\,|P_p^{(-)}(x,y)|^2_p\,(f(x)-f(y))^2\,\widetilde\varrho(x)
\,dv_X(x)\,dv_X(y)\\
=p^{-n+1}\int_{V(\ept)}\int_V\,\Bigg(
\frac{1}{2}\sum_{m=-Mp}^{m=0}\,|P^{(m)}(x,y)|_p^2\,\sum_{k\in\Z}\left|\hat{f}_k(x)
-\hat{f}_{k}(y)\right|^2\\
-\sum_{k\leq 0}\sum_{m=-Mp}^{m=0}\left(\,
 P^{(m)}(x,y).P^{(m+k)}(y,x)-|P^{(m)}(x,y)|_p^2\right)\,\hat{f}_{-k}(x)\,
\overline{\hat{f}_{-k}(y)}\\
-\sum_{k<0}\,\sum_{m=-Mp}^{m=0}\left(\,
 P^{(m+k)}(x,y).P^{(m)}(y,x)-|P^{(m)}(x,y)|_p^2\right)\,\hat{f}_{k}(x)\,
\overline{\hat{f}_{k}(y)}\Bigg)\\
\widetilde\varrho(x)
\,|\xi_x|\,|\xi_y|\,\pi^*dv_{B}(x)\,\pi^*dv_{B}(y)\,.
\end{multline}
We will show that the integral of the first term in the
integrand of \cref{termscomput} is negligible as $p\to\infty$,
while the integrals of the two other terms lead to the boundary term in formula
\cref{Varfla}.
To do so, consider the trivialization
in normal coordinates described in \cref{coordsec},
let us fix $k\leq 0$, and recall that $\mu^{-1}(0)\subset X$
is compact by properness of $\mu\in\cinf(X,\R)$.
By a computation strictly analogous to
\cref{faroffdiagcomput} and using
the estimates \cref{P-sum2,sumcomput1},
%
\cref{wghtedasyth} implies that
for any $\delta>0$, there is $\epsilon_0>0$ such that
for all $\epsilon\in\,]0,\epsilon_0[$, we have
the following asymptotics as $p\to\infty$,
uniform in $x\in\mu^{-1}(0)$,in $k\in\Z$ satisfying $k\leq 0$
and in
$Z,\,Z'\in T^H_xX$ satisfying $|Z|,\,|Z'|<\epsilon\ptheta$,
\begin{equation}\label{finalcomput1}
\begin{split}
&p^{-2n+1}\sum_{m=-Mp}^{m=0}\,P^{(m)}_{p,x}(Z,Z')\,P^{(m+k)}_{p,x}(Z',Z)\\
&=\sum_{m\leq 0}\,
\PP^{(m)}_x(\sqrt{p}Z,\sqrt{p}Z')\PP^{(m+k)}_x(\sqrt{p}Z',\sqrt{p}Z)
+O(p^{-\frac{1}{2}+\delta})\,,
\end{split}
\end{equation}
Now, to deal with the first term of \cref{termscomput}, first note
from \cref{Fouriercoeff} that by standard Fourier theory,
we have the following uniform convergence in $x,\,y\in U$,
\begin{equation}
\sum_{k\in\Z}\left|\hat{f}_k(x)
-\hat{f}_{k}(y)\right|^2=\int_{0}^1\,\big(f(\varphi_t(x))-f(\varphi_t(y))\big)^2\,dt\,.
\end{equation}
Recall the bounded family of
charts \cref{psidef} and formula \cref{dvX=dtdvB}
for the volume form $dv_B$ of the base $(B,g^B)$
defined by \cref{princbdle}. Recalling
the embedding
$\phi:I_0\times V_0\xrightarrow{\sim}V\subset X$ defined by \cref{phidef},
Let $h\in\cinf(V,\R)$ be the positive function satisfying
%
\begin{equation}\label{volh}
\phi^*\pi^*dv_B|_V=:h\,du\,\pi^*dv_{X_0}\,,
\end{equation}
where $du$ is the Lebesgue measure of $\R$,
so that $h|_{V_0}\equiv 1$ by construction.
Finally, note from \cref{Veps} that for any $\epsilon\in\,]0,\epsilon_0[$, we have
\begin{equation}
V(\epsilon)=\{\psi_x(u\,e_1)\in X~|~x\in V_0,\,|u|<\epsilon\}\,.
\end{equation}
Setting $\epsilon'>4\,\epsilon\,\sup_{x\in\mu^{-1}(0)}|\xi_x|$,
using \cref{theta},
the asymptotics \cref{finalcomput1} with $k=0$ and the definition
\cref{rhotilde} of $\widetilde\varrho\in\cinf(X,\R)$,
we then get constants $C,\,C_1,\,C_2>0$ such that for all $p\in\N$,
the following estimates hold,
\begin{equation}\label{Vartermest1}
\begin{split}
&p^{-n+1}\int_{V(\ept)}\int_V
\sum_{m=-Mp}^{m=0}\,|P^{(m)}(x,y)|_p^2\\
&\int_{0}^1\,\big(f(\varphi_t(x))-f(\varphi_t(y))\big)^2\,dt\,
\widetilde\varrho(x)\,|\xi_x|\,|\xi_y|\,\pi^*dv_{B}(y)\,\pi^*dv_{B}(x)
\\
&=p^{-n+1}\int_{V_0}\varrho(\pi(x))\int_{-|\xi_x|\ept}^{|\xi_x|\ept}
\Big(\int_{B^{T^H_xX}(0,\epsilon'\ptheta)}
\sum_{m=-Mp}^{m=0}\,|P^{(m)}_{p,x}(u\,e_1,Z)|_p^2\\
&\int_{S^1}\,(\varphi_t^*f_{x}(u\,e_1)-\varphi_t^*f_{x}(Z))^2\,dt\,
|\xi_{\psi_x(u\,e_1)}|\,|\xi_{\psi_x(Z)}|\,h_x(u\,e_1)\,h_x(Z)
\,dZ\Big)du\,\pi^*dv_{X_0}(x)+O(p^{-\infty})\\
&\leq C p^{-2n+1} \int_{V_0}\varrho(\pi(x)) 
\int_{-|\xi_x|\eptt}^{|\xi_x|\eptt}\\
&\left(\int_{B^{T^H_xX}(0,\epsilon'\pttheta)}
\sum_{m=-Mp}^{m=0}\,|\PP^{(m)}_{x}((u/\sqrt{p})e_1,Z/\sqrt{p})|^2~\,\frac{|u\,e_1-Z|^2}{p}
\,dZ\right)du\,\pi^*dv_{X_0}(x)\\
&\leq  C_1\,p^{-\frac{1}{2}}\int_{V_0}\varrho(\pi(x))
\int_{-|\xi_x|\eptt}^{|\xi_x|\eptt}\Big(\int_{T^H_xX}
\exp\left(-\pi|u\,e_1-Z|^2\right)~\,|u\,e_1-Z|^2
\,dZ\Big)du\,\pi^*dv_{X_0}(x)\\
&\leq C_2\,\ptheta\,.
\end{split}
\end{equation}
This shows that the integral of the first term appearing in the
integrand of \cref{termscomput} tends to $0$ as $p\to\infty$.

To deal with the other terms in formula \cref{termscomput}, 
we will apply the Euler-Maclaurin
formula to the asymptotics \cref{finalcomput1}
as in the proof of \cref{partBergasyth}.
Namely, letting
$I\subset\,]0,+\infty[$ be a compact interval,
we compute the
following
Taylor
expansion in $\frac{k}{\sqrt{p}}$ in the Euler-Maclaurin
formula \cref{EM1}, uniform in $v\in\R$,
$k\in\Z$, $p\in\N$  and $a\in I$,
\begin{equation}\label{EM4}
\begin{split}
&\sum_{m\in\N} e^{-2\pi\left[\big(v-\frac{m}{a\sqrt{p}}\big)^2+\big(v-\frac{m-k}{a\sqrt{p}}\big)^2\right]}\\
&=\int_0^{\infty}
e^{-2\pi\left[\big(v-\frac{t}{a\sqrt{p}}\big)^2+\big(v-\frac{t-k}{a\sqrt{p}}\big)^2\right]}\,dt+a_0\,e^{-2\pi\left[v^2+\big(v-\frac{k}{a\sqrt{p}}\big)^2\right]}+O\Big(\frac{k}{\sqrt{p}}\Big)\\
&=a\sqrt{p}\int_{-\infty}^v e^{-4\pi t^2}\left(1-4\pi t\frac{k}{a\sqrt{p}}+
O\Big(\frac{k^2}{p}\Big)\right)\,dt+a_0\,e^{-4\pi v^2}
+O\Big(\frac{k}{\sqrt{p}}\Big)\\
&=a\sqrt{p}\int_{-\infty}^v\,e^{-4\pi t^2}\,dt+\frac{k}{2}\,
e^{-4\pi v^2}+a_0\,e^{-4\pi v^2}+k\,O\Big(\frac{k}{\sqrt{p}}\Big)\,.
\end{split}
\end{equation}
%
Substracting from formula \cref{EM4} the same formula
for $k=0$, we get the following asymptotics as
$p\to\infty$, uniform in $k\in\Z$, $v\in\R$ and $a\in I$,
\begin{equation}\label{EM4bis}
\sum_{m\in\N} \,e^{-2\pi\left[\big(v-\frac{m}{a\sqrt{p}}\big)^2+\big(v-\frac{m-k}{a\sqrt{p}}\big)^2\right]}
-\sum_{m\in\N} e^{-4\pi\big(v-\frac{m}{a\sqrt{p}}\big)^2}=\frac{k}{2}\,e^{-4\pi v^2}
+k^2\,O(p^{-\frac{1}{2}})\,.
\end{equation}
Recalling the local model \cref{wghtedlocmod}
for the equivariant Bergman kernels,
taking $v=(u+u')/2$ and $a=|\xi_x|$
in \cref{EM4bis} and after a change
of variable $m\mapsto -m$, for any $\delta>0$,
the asymptotics \cref{finalcomput1} give
$\epsilon_0>0$ such that
we get the following asymptotics as $p\to\infty$,
uniform in $x\in\mu^{-1}(0)$, in $k\in\Z$ satisfying $k\leq 0$
and in $Z,\,Z'\in T^H_xX$
as in \cref{Zperp} with $|Z|\,|Z'|<\ept$,
\begin{equation}\label{Varcomputbisbis}
\begin{split}
&p^{-2n+1}\sum_{m=-Mp}^{m=0}\,P^{(m)}_{p,x}(Z,Z')\,P^{(m+k)}_{p,x}(Z',Z)-
\left|P^{(m)}_{p,x}(Z,Z')\right|^2\\
&=\sum_{m\in\N}\,
\left(\PP^{(-m)}_x(Z,Z')\PP^{(-(m-k))}_x(Z',Z)-
|\PP^{(-m)}_x(Z,Z')|^2\right)
+O(p^{-\frac{1}{2}+\delta})\\
&=\frac{k}{|\xi_x|^2}\exp(-\pi p(u+u')^2)\exp\left(-\pi p(u-u')^2\right)\,\Big|\PP_x(\sqrt{p}Z^\perp,\sqrt{p}{Z'}^\perp)\Big|^2+k^2\,O(p^{-\frac{1}{2}+\delta})\,.
\end{split}
\end{equation}
Setting $\epsilon'>4\,\epsilon\,\sup_{x\in\mu^{-1}(0)}|\xi_x|$
as in \cref{Vartermest1},
writing $Z'=u'\,e_1+Z'^\perp\in T_x^HX$ as in \cref{Zperp},
taking Taylor expansions of $f,\,|\xi|$ and $h$
and recalling from \cref{volh} that $h|_{V_0}\equiv 1$,
\cref{theta} together with \cref{Varcomputbisbis}
imply the following asymptotics as $p\to\infty$,
\begin{equation}\label{Vartermest2}
\begin{split}
&p^{-n+1}\int_{V(\ept)}\int_V\widetilde\varrho(x)\,|\xi_x|\,|\xi_y|\\
&
\sum_{m=-Mp}^{m=0}
\left(\,
 P^{(m)}(x,y).P^{(m+k)}(y,x)-|P^{(m)}(x,y)|_p^2\right)
\hat{f}_{k}(x)\,
\overline{\hat{f}_{k}(y)}\,\pi^*dv_{B}(y)\,\pi^*dv_{B}(x)\\
&=p^{n}\int_{V_0}\varrho(\pi(x))\int_{-|\xi_x|\ept}^{|\xi_x|\epsilon\ptheta}\int_{B^{T^H_xX}(0,\epsilon'\ptheta)}
|\xi_{\psi_x(u\,e_1)}|\,|\xi_{\psi_x(Z')}|
\\
&\sum_{m=-Mp}^{m=0}\big(\,
 P^{(m)}_{p,x}(u\,e_1,Z')\,P^{(m+k)}_{p,x}(Z',u\,e_1)
-|P^{(m)}_{p,x}(u\,e_1,Z')|^2_p\big)\\
&\hat{f}_{k,x}(u\,e_1)\,
\overline{\hat{f}_{k,x}(Z')}
\,h_x(u\,e_1)\,h_x(Z')\,dZ'\,du\,\pi^*dv_{X_0}(x)+O(p^{-\infty})\\
&=k\,\int_{V_0}\left(\int_{\R}\int_{T_x^HX}
\exp(-\pi(u+u')^2)\,\exp(-\pi (u-u')^2)\,
\exp(-\pi |Z'^\perp|^2)\,dZ'\,du\right)\\
&\varrho(\pi(x))\,\hat{f}_{k}(x)\,
\overline{\hat{f}_{k}(x)}\,\pi^*dv_{X_0}(x)+
k^2\,O(p^{-\frac{1}{2}+\delta})\\
&=\frac{k}{2}\,\int_{X_0}\varrho(x)~|\hat{f}_{k}(x)|^2\,
dv_{X_0}(x)+k^2\,O(p^{-\frac{1}{2}+\delta})\,.
\end{split}
\end{equation}
Recall now the standard fact from Fourier theory of a smooth
compactly supported
function $f\in\cinf_c(X,\R)$ that for any $N\in\N$,
there exists $C_N>0$ such that for all $x\in U$ and all $k\in\Z$, the associated Fourier coefficient
\cref{Fouriercoeff}
satisfies $|\hat{f}_k(x)|<C_N\,|k|^{-N}$.
Using
\cref{asydiag,Pjprop}, we then get a constant $C>0$ such
that for all $N\geq 2$, $x,\,y\in U$
and $p\in\N$, through the canonical isomorphism
$L^p_x\otimes (L^p_x)^*\simeq\C$, we have
\begin{multline}\label{highFouriermode}
p^{-n+1}\sum_{|k|\geq p^{\frac{1-\epsilon}{8}}}
\sum_{m=-Mp}^{m=0}\left|\left(\,
 P^{(m)}(x,y)\,P^{(m+k)}(y,x)-|P^{(m)}(x,y)|_p^2\right)\,\hat{f}_{k}(x)\,
\overline{\hat{f}_{k}(y)}\right|\\
\leq C\,C_N\,p^{2n+1}\sum_{|k|\geq p^{\frac{1-\epsilon}{8}}}\,|k|^{-2N}\leq C\,C_N\,p^{2n+1-(2N-1)\frac{1-\epsilon}{8}}\,.
\end{multline}
Taking $N\in\N$ large enough, we then get that
\cref{highFouriermode} tends to $0$ as $p\to\infty$,
while on the other hand, we get from \cref{Vartermest2}
the following asymptotics as $p\to\infty$,
\begin{equation}\label{Varestfinal1}
\begin{split}
&p^{-n+1}\sum_{k\geq -p^{\frac{1-\epsilon}{8}}}^{k=0}
\,\int_{V(\ept)}\int_{V}\widetilde\varrho(x)\,|\xi_x|\,|\xi_y|
\,\hat{f}_{k}(x)\,
\overline{\hat{f}_{k}(y)}\\
&\sum_{m=-Mp}^{m=0}\left(
 P^{(m)}(x,y).P^{(m+k)}(y,x)-|P^{(m)}(x,y)|_p^2\right)\,\pi^*dv_{B}(x)\,\pi^*dv_{B}(y)\\
&=\frac{1}{2}\sum_{k\geq -p^{\frac{1-\epsilon}{8}}}^{k=0}\,k\,\int_{X_0}\varrho(x)~|\hat{f}_{k}(x)|^2
dv_{X_0}(x)+\left(\sum_{k\geq -p^{\frac{1-\epsilon}{8}}}^{k=0}
k^2\right)\,O(p^{-\frac{1}{2}+\delta})\\
&=\frac{1}{2}\sum_{k\leq 0}\,k\,\int_{X_0}\varrho(x)~|\hat{f}_{k}(x)|^2
dv_{X_0}(x)+O(p^{-\frac{1}{8}+\frac{3}{8}\epsilon+\delta})\,.
\end{split}
\end{equation}
The same reasonning applied to the last term of formula \cref{Fourierdev2}
gives in turn
the following asymptotic expansion as $p\to+\infty$,
\begin{equation}\label{Varestfinal2}
\begin{split}
&\sum_{k\geq -p^{\frac{1-\epsilon}{8}}}^{k=-1}p^{-n+1}\int_{V(\ept)}
\int_{V(\ept)}\int_{V}\widetilde\varrho(x)\,|\xi_x|\,|\xi_y|
\,\hat{f}_{-k}(x)\,
\overline{\hat{f}_{-k}(y)}\\
&
\sum_{m=-Mp}^{m=0}\left(\,
 P^{(m+k)}(x,y)\,P^{(m)}(y,x)-|P^{(m)}(x,y)|_p^2\right)\,\pi^*dv_{B}(x)\,\pi^*dv_{B}(y)\\
&=\frac{1}{2}\sum_{k< 0}\,k\,\int_{X_0}\varrho(x)~|\hat{f}_{-k}(x)|^2\,dv_{X_0}(x)+O(p^{-\frac{1}{8}+\frac{3}{8}\epsilon+\delta})\,.
\end{split}
\end{equation}
Since $\mu^{-1}(0)\subset X$ is compact by properness of
$\mu\in\cinf(X,\R)$, we can pick a finite cover of $X_0:=\mu^{-1}(0)$
by open sets $\pi(V_0)\subset X_0$ with $V_0\subset\mu^{-1}(0)$
of the form \cref{BNdef} and consider an adapted partition of unity.
Then choosing $\delta>0$ and $\epsilon>0$ small enough
and summing the estimates
\cref{Vartermest1,highFouriermode,Varestfinal1,Varestfinal2}
with $\varrho\in\cinf(X_0,\R)$ running over this partition of unity,
this identifies the second term of the right-hand side of
the asymptotics
\cref{Vardec} with the second term of the right-hand side of
the asymptotics \cref{Varfla}, concluding the proof
of the asymptotics \cref{Varfla} for the variance.

The convergence of the random variable
$N_p^{\alpha}(\NN_p[f]-\IE[\NN_p[f]])$
with $\alpha=\frac{1}{2n}-\frac{1}{2}$
to a centered normal random variable is a
consequence of
the asymptotics \cref{Varfla} for the variance and of
the asymptotics \cref{Expfla} for the expectation established in
\cref{LLN}, thanks to a general argument
adapted from the result of Soshnikov in \cite[Th.\,1]{Sos02}
by Berman in \cite[\S\,6.5]{Ber18}.
This concludes the proof of the Theorem.




\end{proof}



\providecommand{\bysame}{\leavevmode\hbox to3em{\hrulefill}\thinspace}
\providecommand{\MR}{\relax\ifhmode\unskip\space\fi MR }
\providecommand{\MRhref}[2]{%
  \href{http://www.ams.org/mathscinet-getitem?mr=#1}{#2}
}
\providecommand{\href}[2]{#2}

\Addresses

\end{document}